\documentclass[11pt]{amsart}

\usepackage{amscd}
\usepackage{xypic}  
\usepackage{amssymb}
\usepackage{amsthm}
\usepackage{epsfig}
\usepackage{url}
\usepackage{tikz}
\usetikzlibrary{cd}
\usepackage{mathtools}

\usepackage{tikz-cd}
\usepackage{rotating}
\newcommand*{\isoarrow}[1]{\arrow[#1,"\rotatebox{-35}{\(\cong\)}"
]}

\usepackage[all,cmtip]{xy}
\usepackage{amsmath}
\usepackage{color}
\DeclareMathOperator{\coker}{coker}

\newtheorem{thm}{Theorem}[section]  

\newtheorem{lemma}[thm]{Lemma}
\newtheorem{prop}[thm]{Proposition}
\newtheorem{proposition}[thm]{Proposition}
\newtheorem{claim}[thm]{Claim}

\newtheorem{corollary}[thm]{Corollary}

\theoremstyle{definition}

\newtheorem{remark}[thm]{Remark}

\def\ker{\operatorname{ker}}
\def\coker{\operatorname{coker}}
\def\cork{\operatorname{cork}}

\def\im{\operatorname{im}}

\def\c1{\operatorname{c_1}}
\def\c2{\operatorname{c_2}}

\def\Sym{\operatorname{Sym}}

\def\rk{\operatorname{rk}}

\def\ZZ{{\mathbb Z}}

\def\PP{{\mathbb P}}

\def\DD{{\mathbb D}}
\def\SS{{\mathfrak S}}
\def\G{{\mathcal G}}

\def\N{{\mathcal N}}
\def\O{{\mathcal O}}
\def\I{{\mathcal J}}

\def\Z{{\mathcal Z}}
\def\E{{\mathcal E}}

\def\F{{\mathcal F}}
\def\K{{\mathcal K}}
\def\U{{\mathcal U}}
\def\UU{{\mathfrak U}}

\def\Q{{\mathcal Q}} 
\def\K{{\mathcal K}}

\def\HH{\mathfrak{H}}
\def\GG{\mathbb{G}}

\def\c{\mathfrak{c}}

\def\EE{\mathfrak{E}}
\def\MM{\mathfrak{M}}
\def\LL{\mathfrak{L}}

\def\x{\times}                   
\def\cong{\simeq}

\def\+{\oplus}               
\def\*{\otimes}                  

\def\Ext{\operatorname{\rm Ext}}

\def\Shext{\operatorname{ \mathfrak{e}\mathfrak{x}\mathfrak{t} }}
\def\End{\operatorname{End}}
\def\Pic{\operatorname{Pic}}

\def\det{\operatorname{det}}

\def\geq{\geqslant}
\def\leq{\leqslant}

\author[C.~Ciliberto]{Ciro Ciliberto}
\address{Ciro Ciliberto, Dipartimento di Matematica, Universit{\`a} di Roma Tor Vergata, Via della Ricerca Scientifica, 00173 Roma, Italy}
\email{cilibert@mat.uniroma2.it}

\author[F.~Flamini]{Flaminio Flamini}
\address{Flaminio Flamini, Dipartimento di Matematica, Universit{\`a} di Roma Tor Vergata, Via della Ricerca Scientifica, 00173 Roma, Italy} 
\email{flamini@mat.uniroma2.it}

\author[A.~L.~Knutsen]{Andreas Leopold Knutsen}
\address{Andreas Leopold Knutsen, Department of Mathematics, University of Bergen, Postboks 7800,
5020 Bergen, Norway}
\email{andreas.knutsen@math.uib.no}

\title[Elliptic curves and vector bundles on Fano threefolds]{Elliptic curves, $ACM$  bundles and  Ulrich bundles on prime Fano threefolds}
\begin{document}

\maketitle

\begin{abstract} Let $X$ be any smooth prime Fano threefold of degree $2g-2$ in $\PP^{g+1}$, with $g \in \{3,\ldots,10,12\}$. We prove that for any integer $d$ satisfying 
$\left\lfloor \frac{g+3}{2} \right\rfloor \leq d \leq g+3$ the Hilbert scheme parametrizing smooth irreducible elliptic curves of degree $d$ in $X$ is nonempty and has a component of dimension $d$, which is furthermore reduced except for the case when $(g,d)=(4,3)$ and $X$ is contained in a singular quadric. Consequently, we deduce that the moduli space of rank--two slope--stable $ACM$ bundles $\F_d$ on $X$ such that $\det(\F_d)=\O_X(1)$, $c_2(\F_d)\cdot \O_X(1)=d$ and $h^0(\F_d(-1))=0$ is nonempty and has a component of dimension $2d-g-2$, which is furthermore reduced except for the case when $(g,d)=(4,3)$ and $X$ is contained in a singular quadric. This completes the classification of rank--two $ACM$ bundles on prime Fano threefolds. Secondly, we prove that for every $h \in \ZZ^+$ the moduli space
of stable Ulrich bundles $\E$ of rank $2h$ and determinant $\O_X(3h)$ on $X$ is nonempty and has a reduced component of dimension $h^2(g+3)+1$; this result is optimal in the sense that there are no  other Ulrich bundles occurring on $X$.  This in particular 
 shows that any prime Fano threefold is {\em Ulrich wild}. 
 \end{abstract}

\section{Introduction}

 Let  $X \subset \PP^m$ be a smooth irreducible projective variety.   A vector bundle $\E$  on $X$ is said to be \emph{arithmetically Cohen--Macaulay} (\emph{$ACM$} for short) if $h^i(\E(t))=0$  for all $t\in \ZZ$ and all $1 \leq i \leq \dim(X)-1$.
 A vector bundle $\E$  on $X$ is said to be an \emph{Ulrich bundle} if  $h^i(\E(-p))=0$ for all $i \geq 0$ and all $1 \leq p \leq \dim(X)$.   An Ulrich  bundle is always $ACM$ (cf., e.g.,
 \cite[(3.1)]{be2} and semistable (cf. \cite[Thm. 2.9(a)]{CM}). 

 The problem of studying and classifying $ACM$ or Ulrich  bundles is in general a hard and interesting one, and in the last years there has been a lot of activity in this direction  (see, e.g., \cite{be2,ch,Co,CMP} for surveys).  There is however only a short list of varieties for which $ACM$ or Ulrich bundles are completely classified. 
For works regarding $ACM$ and Ulrich bundles on Fano threefolds, mostly of low ranks, we refer to \cite{ac,be2,be0,be3,be1,bf,bf2,bf3,ch,CKL,CM,CFaM1,CFaM2,CFaM3,CFiM,CFK,FP,LMS,LP,mt}.

In this paper we consider the case of smooth prime Fano threefolds of degree $2g-2$ in $\PP^{g+1}$, for which one has $g \in \{3,\ldots,10,12\}$ (see \cite {IP}),  and we study $ACM$ bundles of rank $2$ and Ulrich bundles of any possible rank on them.

Rank--two bundles on threefolds are connected to curves in the threefolds by the famous Hartshorne--Serre correspondence (cf. Proposition \ref{prop:sez}). The bundles  that we will be interested in are related to elliptic curves. Our first main result is:

\begin{thm}\label{thm:ellipt}
Let $X$ be a smooth prime Fano threefold of degree $2g-2$ in $\PP^{g+1}$.
For any $\left\lfloor \frac{g+3}{2}\right\rfloor \leq d \leq g+3$, the Hilbert scheme parametrizing smooth elliptic curves of  degree $d$ in $X$ has an irreducible component of dimension $d$. There is a reduced such component except precisely when $(g,d)=(4,3)$ and $X$ is contained in a singular quadric, in which case the  Hilbert scheme is irreducible and nonreduced with $4$--dimensional tangent space at every point. 

 Furthermore, for $d \leq g+2$ (resp. $d=g+3$) the Hilbert scheme contains points parametrizing elliptic normal curves (resp. non--degenerate elliptic curves). 

\end{thm}

Note that the Hilbert scheme in question is empty if $d< \left\lfloor   \frac{g+3}{2}\right\rfloor$ (cf. Corollary \ref{cor:g4}).

As a consequence, we prove  the following:

\begin{thm} \label{thm:prFanoACM}
  Let $X$ be a smooth prime Fano threefold of degree $2g-2$ in $\PP^{g+1}$.  For any $\Big\lfloor \frac{g+3}{2} \Big\rfloor \leq d \leq g+3$ there exists a slope--stable $ACM$ bundle $\F_d$ on $X$ such that
\begin{equation} \label{eq:propFi}
  \rk(\F_d)=2,\;\; \det(\F_d)=\O_X(1),\;\; c_2(\F_d)\cdot \O_X(1)=d, \;\;  h^0(\F_d(-1))=0. 
\end{equation}
The moduli space of such bundles has an irreducible component of dimension $2d-g-2$, which is reduced except precisely  when $(g,d)=(4,3)$ and $X$ is contained in a singular quadric, in which case the moduli space consists of a single point with a one--dimensional tangent space. 
\end{thm}

 Note that the moduli spaces in the theorem are empty when $d$ is outside the given range (cf. Remark \ref{rem:FM} and Corollary \ref{cor:g4}).

For $d> \Big\lfloor \frac{g+3}{2} \Big\rfloor$ the latter  result is an improvement of \cite[Thm. 3.1]{bf}, since our result avoids the assumption that $X$ be {\it ordinary} (that is, $X$ contains a line $\ell$ with normal bundle $\N_{\ell/X} \cong \O_{\ell} \+ \O_{\ell}(-1)$). For $d=\Big\lfloor \frac{g+3}{2} \Big\rfloor$ our statement is contained in 
\cite[Thm. 3.1-3.2]{bf}, except that we avoid again the assumption that $X$ be ordinary when $(g,d)=(3,3)$; also note that the nonemptiness statement
of the moduli spaces in question when $d=\Big\lfloor \frac{g+3}{2} \Big\rfloor$  has until now only been known by a case--by--case treatment in various papers (see \cite[\S 3.1.1]{bf} and references therein), and that our proof via the existence of elliptic curves as in Theorem \ref{thm:ellipt} provides an alternative, unified and simpler proof of this fact scattered in literature. In particular, our proof is independent of the one in \cite{bf}.

The main motivation behind Theorem \ref{thm:prFanoACM} is that it completes the classification of rank--two $ACM$ bundles on smooth prime Fano threefolds. Indeed, since the property of being $ACM$ is independent of twists, one may restrict the study of such bundles to the ones with determinants $\O_X(m)$  for $m \in \{0,1\}$. Combining \cite[Thm. p.~114]{bf} and Theorem \ref{thm:prFanoACM} one obtains: {\it A smooth prime Fano threefold $X$ of degree $2g-2$ in $\PP^{g+1}$ carries a rank--two $ACM$ bundle $\F$ such that $c_1(\F)=[\O_X(m)]$ with $m \in \{ 0,1\}$ if and only if $d:=c_2(\F) \cdot \O_X(1)$ satisfies:
\begin{itemize}
  \item[(i)] $m=1$, $d =1$ or $\Big \lfloor \frac{g+3}{2} \Big \rfloor \leq d \leq g+3$,
  \item[(ii)] $m=0$, $d \in \{2,4\}$.
\end{itemize}}
\noindent The new contribution is the abolition of the assumption that $X$ be ordinary when $d \geq \frac{g}{2}+2$ in the ``if'' part. 

The proof of Theorem \ref{thm:ellipt} occupies \S \ref{sec:pf1}. The idea of the proof is conceptually simple: We exploit Mukai's description of the $K3$ hyperplane sections of prime Fano threefolds and the surjectivity of the period map of $K3$ surfaces to prove that there always exists a codimension--one family of hyperplane sections containing a smooth elliptic curve of degree $\Big \lfloor \frac{g+3}{2} \Big \rfloor$. We then construct the desired elliptic curves of higher degree by smoothing the union of an elliptic curve and a line intersecting transversally in one point. Despite the simple idea of the proof, there are several  technical details to be fixed, and most of this work is carried out in \S \ref{sec:ellvb}. The proof of Theorem \ref{thm:prFanoACM} follows as a rather direct consequence of  Theorem \ref{thm:ellipt} and is carried out in \S \ref{sec:pf2}.

Our second main result concerns Ulrich bundles:

\begin{thm} \label{thm:prFano}
  Let $X$ be a smooth prime Fano threefold of degree $2g-2$ in $\PP^{g+1}$. Then, for every  $h \in \ZZ^+$, there exists a stable\footnote{For Ulrich bundles the notions of stability and slope--stability coincide (cf. \cite[Thm. 2.9(a)]{CM}).}  Ulrich bundle $\E$ of rank $2h$ and determinant $\O_X(3h)$ on $X$. The moduli space of such bundles has a reduced and irreducible component of dimension $h^2(g+3)+1$.
\end{thm}

This result is optimal in the sense that a prime Fano threefold $X$ does not carry any Ulrich bundles of odd rank by \cite[Cor. 3.7]{CFK} and any Ulrich bundle of rank $2h$ on $X$ is forced to have determinant $\O_X(3h)$ by \cite[Thm. 3.5]{CFK}.  

Theorem \ref {thm:prFano} in particular shows that any  smooth prime Fano threefold of degree $2g-2$ in $\PP^{g+1}$ is \emph{Ulrich wild}\footnote{In analogy with a  definition in \cite {DG}, a variety $X$ is said to be Ulrich wild if it possesses families of dimension $n$ of pairwise non--isomorphic, indecomposable, Ulrich bundles for arbitrarily large $n$.}.  Only very few cases of varieties carrying stable Ulrich bundles of infinitely many ranks are hitherto known.

The idea of the proof of Theorem \ref {thm:prFano} is also conceptually quite simple. Rank--two Ulrich bundles are obtained  twisting by $\O_X(1)$ the rank--two $ACM$ bundles in Theorem \ref{thm:prFanoACM} with $d=g+3$. Then Ulrich bundles of all even ranks $r \geq 4$ are obtained by induction as deformations of extensions of rank $r-2$ Ulrich  bundles with rank $2$ Ulrich  bundles. The proof is carried out in \S \ref{sec:pf3}. 

Finally, in \S \ref {sec:final}, taking profit from a correspondence we found in \cite [Thm.  3.5]{CFK} between rank two Ulrich bundles and certain curves on prime Fano threefolds, we propose an application to the moduli space of curves of genus $5g$ with  theta--characteristics with $g+2$ independent global sections, for $3\leq g\leq 5$.

 \medskip
Throughout the paper we work over the field of complex numbers.
 
\medskip

{\bf Acknowledgements:} C.~Ciliberto and F.~Flamini are members of  GNSAGA of the Istituto Nazionale di Alta Matematica ``F. Severi"  and acknowledge support from  the MIUR Excellence Department Project awarded to the Department of Mathematics, University of Rome Tor Vergata, 
CUP E83C18000100006.  A.~L.~Knutsen acknowledges support from the Meltzer Research Fund, the Trond Mohn Foundation Project ``Pure Mathematics in Norway'' and  grant 261756 of the
Research Council of Norway.

\section{Basics on prime Fano threefolds}\label{sec:fano}


A \emph{ (smooth) prime Fano threefold (of index 1) of degree $2g-2$ in $\PP^{g+1}$} is a smooth, \linebreak irreducible, projective, non--degenerate threefold $X \subset \PP^{g+1}$ of degree $2g-2$ 
such that  \linebreak $\omega_X \cong \O_X(-1)$ and 
$\Pic(X) \cong \ZZ[\O_X(1)]$. These threefolds exist only for $g \in \{3,\ldots,10,12\}$ (see \cite{IP}) and $g$ is called the \emph{genus} of the threefold. Any such $X$ is projectively normal. Moreover, its general curve section is a canonical curve of genus $g$ and its general hyperplane section is a K3 surface $S$ of genus $g$, which satisfies $\Pic(S) \cong  \ZZ[\O_S(1)]$  by Noether--Lefschetz (cf. \cite[Thm.\,3.33]{Vo}).

Prime Fano threefolds of genus $g=3,4,5$ are complete intersections: for $g=3$ they are quartic hypersurfaces in $\PP^4$, for $g=4$ they are complete intersections of type $(2,3)$ in $\PP^5$ and for $g=5$ they are complete intersections of type $(2,2,2)$ in $\PP^6$. For $g\geq 6$ the description of prime Fano threefolds is more complicated (see  \cite {IP}). If $g=4$, we will say that $X$ is in case $(\star)_4$ if the unique quadric containing $X$ is singular.

The Hilbert scheme $\LL(X)$ of lines in a prime Fano threefold $X$ is one--dimensional (cf., e.g., \cite[Cor. 1]{Is2}), and the lines in $X$ sweep out a surface that we will denote by $R_1(X)$. The normal bundle of any line $\ell$ in $X$ satisfies
\[ \N_{\ell/X} \cong \O_{\ell} \+ \O_{\ell}(-1) \;\; \mbox{or} \;\; \O_{\ell}(1) \+ \O_{\ell}(-2)\]
(cf. \cite[Lemma 3.2]{Is} and \cite[Lemma 1]{Is2}), so that in particular
\begin{equation}
  \label{eq:lines}
  \chi(\N_{\ell/X})=1 \; \; \mbox{and} \;\; h^0(\N_{\ell/X}) \leq 2.
\end{equation}
As recalled in the introduction, $X$ is said to be {\it ordinary}\footnote{Only two instances of non--ordinary  prime Fano threefolds are known: the {\it Fermat quartic} for $g=3$ and an example of Mukai--Umemura for $g=12$ \cite{MU}. It is furthermore known that a {\it general} prime Fano threefold of any genus is ordinary.}  if it contains a line $\ell$ with normal bundle $\N_{\ell/X} \cong \O_{\ell} \+ \O_{\ell}(-1)$.

We will make use of the following result:

\begin{prop}\label{prop:compR}
  Let $X$ be a prime Fano threefold of degree $2g-2$ in $\PP^{g+1}$. One of the following occurs:
  \begin{itemize}
\item[(I)] $R_1(X)$ has a component that is not a cone; or
\item[(II)] $R_1(X)$ consists of $40$ irreducible components that are all cones, and general pairs of lines in different components do not intersect.
\end{itemize}
The latter case\footnote{The {\it Fermat quartic} is an example, and it is not known whether there are others (see again \cite{HT}).}
 occurs only for $g=3$.
\end{prop}

\begin{proof}
This is proved for $g \geq 4$ in \cite[Thm.~3.4(iii)]{Is}\footnote{Note that
  the statement \cite[Thm.~3.4(ii)]{Is} is incorrect; the correct statement is \cite[Prop. 1]{Is2}.}, and for $g=3$ in \cite[p. 496]{HT}.
\end{proof}

We also record here the following information that will be necessary for us. 

\begin{lemma}\label{lem:dual}
Let $X$ be a prime Fano threefold of degree $2g-2$ in $\PP^{g+1}$. Then its dual variety $X^*$ is a hypersurface. As a consequence we have:
\begin{itemize}
\item [(i)] a general tangent hyperplane to $X$ at a general point $x\in X$ is simply tangent only at $x$;
\item [(ii)] the general member of any $g$--dimensional family of singular hyperplane sections of $X$ has a rational double point at a general point of $X$.
\end{itemize}
\end{lemma}

 \begin{proof} If $X^*$ is not a hypersurface, then the general tangent hyperplane to $X$ at a general point $x\in X$ is tangent along a linear subspace of positive dimension containing $x$ (see \cite[Thm. 1.5.10]{Rus}), contradicting the fact that $\dim(\LL(X))=1$. Assertion (i) follows again from \cite[Thm. 1.5.10]{Rus} and assertion (ii) follows from (i). \end{proof}

 \section{Elliptic curves and rank--two bundles} \label{sec:ellvb}

 Throughout this section $X$ will denote an arbitrary prime  Fano threefold of degree $2g-2$ in $\PP^{g+1}$. 
We gather some useful  results concerning rank--two vector bundles and elliptic curves on $X$. The starting point is the following:

\begin{prop} \label{prop:sez} (i) Any   local complete intersection  curve $E_d \subset X$ of degree $d$ with trivial canonical bundle gives rise to an exact sequence
\begin{equation}
  \label{eq:HS1}
  \xymatrix{
    0 \ar[r] & \O_X \ar[r]^{s_d} & \F_{d} \ar[r] & \I_{E_{d}/X}(1) \ar[r] & 0,} 
\end{equation} for a unique  vector bundle $\F_d$ on $X$ satisfying
\begin{equation}
  \label{eq:propF}
  \rk(\F_d)=2,\;\; \det(\F_d)=\O_X(1),\;\; c_2(\F_d)\cdot \O_X(1)=d, \;\; h^0(\F_d(-1))=0.
\end{equation}

\smallskip

(ii) Conversely, if $\F_d$ is a vector bundle on $X$ satisfying \eqref{eq:propF}, then it is slope--stable and any  non--zero  section $s_d \in H^0(\F_d)$ vanishes along a local complete intersection  curve $E_d$ of degree $d$ with trivial canonical bundle and there exists a sequence like \eqref{eq:HS1}. Moreover,  non--proportional non--zero sections vanish along distinct curves.

\smallskip

(iii) If, in (ii), we  furthermore have $h^0(\F_d) \geq 2$, then 
we have a natural map
\begin{eqnarray}
  \label{eq:mappawedge}
  w: \mbox{\rm{Grass}}(2,H^0(\F_d))  & \longrightarrow & |\O_X(1)|, \\
 \nonumber       V & \mapsto & Z(s_1 \wedge s_2),
\end{eqnarray}
where $\langle s_1,s_2 \rangle$ is any basis for $V$  and $Z(s_1 \wedge s_2)$ is the zero scheme of $s_1 \wedge s_2$. 
The surface  $Z(s_1 \wedge s_2)$  contains the pencil $\PP(V)$ of curves  that are zero schemes $Z(us_1+vs_2)$ of sections of the form $us_1+vs_2$  for all $(u:v)\in \PP^1$  and it is singular at least where $s_1$ and $s_2$ both vanish.
\end{prop}

\begin{proof}
  The correspondence in (i)--(ii) between the curve $E_d$ and the vector bundle $\F_d$ with section $s_d$ is well--known,  by the Hartshorne--Serre correspondence (see, e.g., \cite{har,har2,A}).
Properties \eqref{eq:propF} can be deduced from \eqref{eq:HS1}; moreover, the 
facts that $h^0(\F_d(-1))=0$ and and $\Pic(X) \cong \ZZ[\O_X(1)]$  imply that {\it any} section of $\F_d$ vanishes along a curve. Finally, slope--stability of $\F_d$ is an immediate consequence of the facts that $\Pic(X) \cong \ZZ[\O_X(1)]$, $\rk(\F_d)=2$, $\det(\F_d) \cong \O_X(1)$ and $h^0(\F_d(-1))=0$.
  
Assertion (iii) is standard.
\end{proof}

\begin{remark}\label{rem:local}  Assume that in case (iii) of Proposition \ref{prop:sez} the zero schemes $E_i:=Z(s_i)$, $i=1,2$, are smooth, irreducible, distinct, elliptic curves. Let $Z$ be the intersection scheme $E_1 \cap E_2$. Then the surface $S:=Z(s_1 \wedge s_2)$ is singular along $Z$. More precisely, suppose that $E_1$ and $E_2$ have a contact of order $k$ at a point $p$, so that $Z$ locally at $p$ is a curvilinear scheme of length $k$. We can introduce local coordinates $(x,y,z)$ on $X$ around $p$ so that $E_1$ and $E_2$ have local equations $x=y=0$ and $x-y=y+z^k=0$, respectively. Then $S$ has equation  $xz^k +  y^2=0$  around $p$, so that $S$  has a double point at $p$ and actually it has $k$ double points infinitely near to $p$ along $E_1$ or, equivalently,  along $E_2$, or, if we prefer, along $Z$. Therefore, on a desingularization $S'$ of $S$ the  strict transforms  $E'_1$ and $E'_2$  of $E_1$ and $E_2$  are isomorphic to $E_1$ and $E_2$ respectively, and belong to a base point free pencil, hence ${E'}_1^2={E'}_2^2 =0$.  
\end{remark}

Denote by $\overline{\HH}_d(X)$ the Hilbert scheme parametrizing local complete intersection curves in $X$ of degree $d$ with trivial canonical bundle and by $\HH_d(X)$ the sublocus parametrizing smooth irreducible (elliptic) curves. Recall that the tangent space of $\overline{\HH}_d(X)$ at a point $[E_d]$ is $H^0(\N_{E_d/X})$, by standard theory of Hilbert schemes. 

Denote by $\MM_d(X)$ the moduli space of  vector  bundles on $X$
satisfying \eqref{eq:propF}, which are all slope--stable by Proposition \ref{prop:sez}(ii).
By the same proposition there is a  morphism
\begin{equation} \label{eq:mappap}
  p: \overline{\HH}_d(X) \longrightarrow \MM_d(X)
  \end{equation}
mapping $[E_d]$ to $[\F_d]$. The image consists of the vector bundles having a section (by Proposition \ref{prop:sez}(ii)), and  the fiber over a point $[\F_d]$ is $\PP(H^0(X,\F_d))$.

The next three lemmas collect useful properties of the members of $\overline{\HH}_d(X)$ and $\MM_d(X)$:

\begin{lemma} \label{lemma:FM} 
  Let $[\F_d] \in \MM_d(X)$. Then
  \begin{itemize}
  \item[(i)] $\chi(\F_d)=g+3-d$ and $h^3(\F_d)=0$;
    \item[(ii)] $\chi(\F_{d} \* \F_{d}^*)=  g-2d+3$ and 
$h^0(\F_{d} \* \F_{d}^*)=1$.
  \end{itemize}
\end{lemma}

\begin{proof}
  The computations of $\chi$ follow from Riemann--Roch on $X$. Since $\rk(\F_d)=2$ and $\det(\F_d)=\O_X(1)$, we have $\F_d^* \cong \F_d(-1)$, whence $h^3(\F_d)=h^0(\F_d^*(-1))=h^0(\F_d(-2))=0$. Finally, $h^0(\F_{d} \* \F_{d}^*)=\dim (\End(\F_d))=1$ as $\F_d$ is slope--stable, whence simple. 
\end{proof}

\begin{remark} \label{rem:FM}
  If a bundle $\F_d$  in $\MM_d(X)$ is $ACM$, then $h^0(\F_d)=\chi(\F_d)=g+3-d$ by Lemma \ref{lemma:FM}(i). Consequently, the locus of $ACM$ bundles in $\MM_d(X)$ is empty if $d > g+3$.
\end{remark}

\begin{lemma} \label{lemma:h3h2}
  Let $[E_d] \in \overline{\HH}_d(X)$ and $[\F_d]=p([E_d]) \in \MM_d(X)$. Then:

  \begin{itemize}
  \item[(i)] $\N_{E_d/X} \cong \F_d|_{E_d}$ and $\chi(\N_{E_d/X})=d$;
  \item[(ii)] $h^i(\I_{E_d/X})=0$ for $i \in \{0,1,3\}$ and $h^2(\I_{E_d/X})=1$;
   \item[(iii)] $\chi(\I_{E_d/X}(1))=g+2-d$ and $h^i(\I_{E_d/X}(1))=0$ for $i \in \{2,3\}$;
   \item[(iv)]  $h^i(\F_d(-1))=0$ for $i \in \{0,1,2,3\}$;
   \item[(v)] 
     $h^0(\F_d)=h^0(\I_{E_d/X}(1))+1$, $h^1(\F_d)=h^1(\I_{E_d/X}(1))$
     and $h^2(\F_d)=0$; 
\item[(vi)] 
  $h^3(\F_{d} \* \F_{d}^*)=0$, and $h^2(\F_{d} \* \F_{d}^*) \leq h^1(\N_{E_d/X})$, with equality if $h^1(\F_d)=0$.
\end{itemize}
  \end{lemma}

\begin{proof}  
Property (i) is immediate from \eqref{eq:HS1} and Riemann--Roch. From the  sequence
  \begin{equation}
  \label{eq:IS1}
  \xymatrix{
    0 \ar[r] & \I_{E_d/X} \ar[r] & \O_X \ar[r] & \O_{E_d} \ar[r] & 0} 
\end{equation}
we deduce (ii) and (iii). 
From \eqref{eq:HS1} and (iii) we deduce (v), whereas from 
\eqref{eq:HS1} and (ii) we deduce that $h^i(\F_d(-1))=0$ for $i \in \{0,1\}$. Since $\F_d^* \cong \F_d(-1)$, Serre duality yields that $h^i(\F_d(-1))=h^{3-i}(\F_d^*)=h^{3-i}(\F_d(-1))=0$ for $i \in \{2,3\}$ as well, proving (iv).

To prove (vi), tensor  \eqref{eq:HS1} by
  $\F_{d}^*\cong \F_{d}(-1)$ to obtain
\[
    \xymatrix{
      0 \ar[r] & \F_{d}(-1) \ar[r] & \F_{d} \* \F_{d}^* \ar[r] & \F_{d} \* \I_{E_{d}/X} \ar[r] & 0.}
 \]
    By (iv) we get
    $h^i(\F_{d} \* \F_{d}^*)=h^i(\F_d \* \I_{E_d/X})$ for all $i$.
Now (vi) follows from 
\eqref{eq:IS1} tensored by $\F_d$, properties (i) and (v) and Lemma \ref {lemma:FM}(i).
\end{proof}

\begin{remark} \label{rem:elln}
  By \eqref{eq:IS1} twisted by $\O_X(1)$ and linear normality of $X$ we see that $E_d$ spans a $\PP^{d-1-h^1(\I_{E_d/X}(1))}$ in its embedding in $\PP^{g+1}$. Hence, $E_d$ is projectively normal if and only if $h^1(\I_{E_d/X}(1))=0$
  (which can only happen for $d \leq g+2$) and nondegenerate if and only if
  $h^1(\I_{E_d/X}(1))=d-g-2$  (which can only happen for $d \geq g+2$).
\end{remark}

\begin{lemma} \label{lemma:normalb}
  Let $[E_d] \in \HH_d(X)$. Then the following conditions are equivalent:
  \begin{itemize}
  \item[(i)] $h^1(\N_{E_d/X}) > 0$,
  \item[(ii)] $h^1(\N_{E_d/X}) =1$,
    \item[(iii)] $h^0(\N_{E_d/X}) > d$,
  \item[(iv)] $\N_{E_d/X} \cong \O_{E_d} \+ \O_{E_d}(1)$,
  \item[(v)] there is a section of $\N_{{E_d}/X}$ vanishing along a scheme of length $d$. 
 \end{itemize}
\end{lemma}

\begin{proof}
  We have $h^0(\N_{{E_d}/X}) = \chi(\N_{{E_d}/X})+h^1(\N_{{E_d}/X})=d+h^1(\N_{{E_d}/X})$ by Lemma \ref{lemma:h3h2}(i), proving that  (i) and (iii) are equivalent. Moreover, for any $x \in {E_d}$, we have $h^0(\N_{{E_d}/X}(-x)) \geq d-2 >0$, whence there is a nonzero section  $s: \O_{E_d} \to \N_{{E_d}/X}$ vanishing at $x$. Letting $Z$ denote its scheme of zeros, and saturating, we obtain 
  \begin{equation} \label{eq:sticaz}
    \xymatrix{
0 \ar[r] & \O_{E_d}(Z)  \ar[r] & \N_{{E_d}/X} \ar[r] & \O_{E_d}(1)(-Z) \ar[r] & 0.
      }
  \end{equation}
If $\deg(Z) <d$, we thus have $h^1(\N_{{E_d}/X})=0$. This proves that (i) implies (v).

Assume (v) and let $Z$ be the length $d$ scheme along which a section $s$ of
$\N_{{E_d}/X}$ vanishes. Then $Z \in |\O_{E_d}(1)|$ and \eqref{eq:sticaz} becomes
\[
    \xymatrix{
0 \ar[r] & \O_{E_d}(1) \ar[r] & \N_{{E_d}/X} \ar[r] & \O_{E_d} \ar[r] & 0,
      }
\]
which must split, as $\dim (\Ext^1(\O_{E_d},\O_{E_d}(1)))=h^1(\O_{E_d}(1))=0$. This proves that (v) implies (iv). Clearly, (iv) implies (ii), and (ii) implies (i). 
\end{proof}

We denote by  $\SS_d(X)$ the locus of hyperplane sections of $X$ of the form $Z(s_1 \wedge s_2)$ for (non--proportional) $s_1,s_2 \in H^0(\F_d)$ for some $[\F_d] \in \MM_d(X)$. Note that each member of $\SS_d(X)$ contains a pencil of elements from  $\overline{\HH}_d(X)$ by Proposition \ref{prop:sez}(iii). The following result is crucial in the proof of Theorem \ref{thm:ellipt}.

\begin{prop} \label{prop:g4}
  If $d \leq g+1$, then
    \begin{itemize}
    \item[(i)]  the map $p$ is surjective and $\MM_d(X)$, $\overline{\HH}_d(X)$ and $\SS_d(X)$ are equidimensional of dimensions 
\[ \dim(\MM_d(X))= 2d-g-2, \;\; \dim(\overline{\HH}_d(X))=d \;\; \mbox{and} \;\; \dim(\SS_d(X))=g;\]
 \item[(ii)] for general $E_d$ in any component of $\overline{\HH}_d(X)$ and general $\F_d$ in any component of $\MM_d(X)$ we have
  \[ h^0(\F_d)=g+3-d, \; \;  h^1(\F_d)=0 \;\; \mbox{and}\] 
	\[h^0(\I_{E_d/X}(1))=g+2-d, \; \; h^1(\I_{E_d/X}(1))=0;\] 
\item[(iii)] the general $S$ in any component of $\SS_d(X)$ has at worst a rational double point, which is  a general point of $X$, and contains a pencil of members of $\overline{\HH}_d(X)$, all  zero loci of non--zero  sections of the same element of $\MM_d(X)$, having at most a single base point at the rational double point of $S$;
  \item[(iv)] for any component $\HH_d(X)'$ of $\overline{\HH}_d(X)$, the locus
\[ \{ x \in R_1(X) \; | \; x \in E_d \; \; \mbox{for some  irreducible $E_d$ in    $\HH_d(X)'$   not contained in $R_1(X)$} \} \]
 is dense in every component of $R_1(X)$;
\item[(v)] the general $E_d$   in any component of   $\overline{\HH}_d(X)$ intersects $R_1(X)$ transversally, such that through each intersection point there passes a unique line $\ell$  contained in $X$  (hence this line intersects $E_d$ transversally in one point),  and with the property that there is a hyperplane section of $X$ containing $E_d$ but not $\ell$; 
\item[(vi)] the general $E_d$ in any {\it reduced} component of $\HH_d(X)$ intersects $R_1(X)$ (transversally) in at least two points through which there pass two non-intersecting lines.
\end{itemize}
\end{prop}

 \begin{proof} 
   (i)-(ii)   For any $[\F_d] \in \im(p)$ we have $h^2(\F_d)=0$ by
Lemma \ref{lemma:h3h2}(v), and this also holds for a general 
   $[\F_d] \in \MM_d(X)$ by semicontinuity. Using 
Lemma \ref{lemma:FM}(i) we thus get
      \begin{equation}
     \label{eq:dimhfd}
     h^0(\F_d)=\chi(\F_d)+h^1(\F_d)-h^2(\F_d)+h^3(\F_d)=g+3-d+h^1(\F_d) \geq g+3-d \geq 2,
     \end{equation}
     whence $p$ is surjective by Proposition \ref{prop:sez}(ii). For any component $\MM'$ of $\MM_d(X)$, one has, by standard deformation theory of vector bundles and Lemmas \ref{lemma:FM}(ii) and \ref{lemma:h3h2}(vi), 
\begin{eqnarray}
 \label{eq:dimM}   \dim(\MM') & \geq & h^1(\F_{d} \* \F_{d}^*)-h^2(\F_{d} \* \F_{d}^*) \\
 \nonumber  & = & -\chi(\F_{d} \* \F_{d}^*)+h^0(\F_{d} \* \F_{d}^*)-h^3(\F_{d} \* \F_{d}^*)=2d-g-2.
 \end{eqnarray}

  Let us pretend that $\MM_d(X)$ is a fine moduli space (if not one argues similarly at a local level). Then we can consider the variety $\GG$ with a morphism $\pi: \GG\to \MM_d(X)$, such that for each $[\F_d]\in \MM_d(X)$, the fiber of $\pi$ over $[\F_d]$ is $ {\rm  Grass}(2, H^0(\F_d))$, which has dimension at least $2(g+1-d)$ by \eqref{eq:dimhfd}. 
Together with \eqref{eq:dimM}, this implies
\begin{equation}\label{eq:probt}
   \dim(\GG') \geq (2d-g-2)+2(g+1-d)=g,
\end{equation} for any component $\GG'$ of $\GG$. We can define a universal version of the map $w$ as in \eqref{eq:mappawedge}: 
   \begin{eqnarray}\label{eq:mappawedge-univ}
  \mathfrak{w}:  \GG    & \longrightarrow & |\O_X(1)|, \\
   \nonumber      \left([\F_d],[V=\langle s_1,s_2\rangle]\right)  & \mapsto & Z(s_1 \wedge s_2).
\end{eqnarray}
 By definition of $\SS_d(X)$, the  image of the map $\mathfrak{w}$ is  $\SS_d(X)$.  Note that  every component of $\SS_d(X)$ has dimension at most $g$, as $\Pic(S)\cong \ZZ[\O_S(1)]$ for the general   $S \in |\O_X(1)|$.

We claim that

\begin{equation} \label{eq:wgf}
\mbox{$\mathfrak{w}$ is generically finite  on any component of $\GG$.}
\end{equation}
To prove this, we argue by contradiction and assume that $\mathfrak{w}$ is not generically finite. Let $[S]$ be general in the image of $\mathfrak{w}$. Then there is a positive--dimensional family of  pairs $\left([\F_d], [V=\langle s_1,s_2\rangle]\right)\in \mathbb G$  such that $S=Z(s_1 \wedge s_2)$ and $S$ has a positive--dimensional family of pencils $\PP(V)$ as in Proposition \ref {prop:sez}(iii). Accordingly, $S$ has a family  $\EE$  of dimension at least 2 of curves, the general one $E$ being  a smooth irreducible elliptic curve moving in a pencil.
 By Remark \ref {rem:local}, the strict transform $E'$ of $E$ on the minimal desingularization of $S$ satisfies $E'^2=0$, which  contradicts the fact that  $\dim(\EE) \geq 2$. We have thus proved \eqref{eq:wgf}.

From \eqref{eq:wgf} it follows that 
$\dim(\SS')=\dim(\GG')= g$ for any component $\SS'$ of $\SS_d(X)$ and $\GG'$ of $\GG$ (by \eqref {eq:probt}), and as a consequence, also that  for any component $\MM'$ of $\MM_d(X)$ one has $\dim(\MM')=2d-g-2$ and $h^0(\F_d)=g+3-d$ for general $[\F_d] \in \MM'$. Considering the map $p$ in \eqref{eq:mappap}, we see that it also follows that
\[
\dim(\HH')=\dim(\MM')+\dim(\PP(H^0(\F_d))=(2d-g-2)+(g+2-d)=d,
\]
 for any component
$\HH'$ of $\overline{\HH}_d(X)$,  
which finishes the proof of (i).   At the same time,
\eqref{eq:dimhfd} shows that $h^1(\F_d)=0$, and Lemma \ref{lemma:h3h2}(v)  yields that  $h^0(\I_{E_d/X}(1))=g+2-d$ and   $h^1(\I_{E_d/X}(1))=0$ for general $[E_d] \in \HH'$, which finishes the proof of (ii).

(iii) If the general member $S$ of   any component of   $\SS_d(X)$ is singular, then by (i) and Lemma \ref {lem:dual}(ii), it has a rational double point at a general point of $X$.   Then (iii) follows from Proposition \ref{prop:sez}(iii).

(iv)   Let $R$ be an irreducible component of $R_1(X)$ and $S$ be a general member of a component $\SS'$ of $\SS_d(X)$.  
Note that, since the possible rational double point of $S$ is a general point of $X$, it does not lie on $R$.
 
We  claim  that $C:=S \cap R$ is reduced. Suppose this is not the case; then the hyperplane $\Pi$ containing $S$ would be tangent to $R$ and, since $\dim(\SS')=g$, it would be the general hyperplane tangent to $R$ so the dual variety $R^*$ of $R$ would be a hypersurface. On the other hand, since the curve cut out by $\Pi$ on $R$ is non--reduced, this would imply  that  the general tangent hyperplane to $R$ is tangent along a curve,  whence the dual variety $R^*$ of $R$ would not be a hypersurface, a contradiction. 

  Next we claim that the family of curves $\{ S \cap R \; | \; [S] \in \SS', S \neq R\}$ covers $R$. Indeed, if not, we would have $\SS' \subseteq |\I_{C/X}(1)|$, so that both would be $g$--dimensional by (i). Since $\Pic(X) \cong \ZZ[\O_X(1)]$ we have $R \in |\O_X(m)|$ for some $m \in \ZZ^+$. Then 
\[
  \xymatrix{
0 \ar[r] & \I_{S/X}(1) \cong \O_X \ar[r] & \I_{C/X}(1) \ar[r] & \I_{C/S}(1) \cong \O_S(1-m) \ar[r] & 0 
    }
  \]
shows that $\dim|\I_{C/X}(1)|\leq 1$, a contradiction.  

  Since $E_d$ moves in a pencil on $S$ and has no base point on $C$ by (iii), we see that $E_d$ passes through the general point of $R$, proving (iv). 

(v) Since through the general point of $R$ there passes only one line, and $S$ contains at most finitely many lines (because its resolution of singularities is a K3 surface by (iii)), we can by (iv) assume that the locus
\[ C^{\circ}=\{x \in C \; | \; \mbox{$C$ is smooth at $x$ and there passes a unique line $\ell$ through $x$ and $\ell \not \subset S$}\} \]
is the whole $C$ but finitely many points. Since the possible base point of the pencil of curves $|E_d|$ on $S$ lies outside $R$, whence off $C$, the general member $E_d$ of the pencil intersects $C$ in points of  $C^{\circ}$ and transversally. Thus the intersection $E_d \cap R$ is transverse, and through each intersection point there passes a unique line $\ell$ that is not contained in $S$. Since $\ell \cdot S=\ell \cdot \O_X(1)=1$, we have that $\ell \cap E_d$ is transverse and consists of only one point. This proves (v). 

(vi)   If we are in case (II) of Proposition \ref{prop:compR}, the result follows from (v) and the fact that lines in different components of $R_1(X)$ do not intersect.

Assume therefore that we are in case (I) of Proposition \ref{prop:compR}, so that we can pick a component $R$ of $R_1(X)$ that is not a cone. By assumption there is a reduced component of $\HH_d(X)$, which we denote by $\HH'$, and we let $[E_d] \in \HH'$ be general.  
Since $\Pic(X) \cong \ZZ[\O_X(1)]$, we have $R \in |\O_X(m)|$ for some $m \geq 1$. By (v), the intersection $E_d \cap R$ is transversal for general $[E_d] \in \HH'$ and occurs at points through which there passes a unique line in $X$, which each intersects $E_d$ in only that point and transversally (in particular the lines are distinct). It follows that $E_d \cap R$ consists of $md \geq 3$ points.

  Denote by  $\HH^{\circ}$ the dense open subset of $\HH'$ parametrizing curves not contained in $R$ and by $\Sym^2(R)^{\circ}$ the dense open subset  
of  $\Sym^2(R)$  consisting of all elements $x+y$ such that $x \neq y$ and there passes a unique line in $R$ through both $x$ and $y$ and these lines are distinct. Consider the incidence variety
\[ J_d:=\{ [(E_d,x+y)] \in \HH^{\circ} \x  \Sym^2(R)^{\circ} \;|\; x,y \in E_d \cap R\},\]
 with projections $p_d:J_d \to  \HH^{\circ}$ and $q_d: J_d \to  \Sym^2(R)^{\circ}$.
By (v) we have
$J_d \neq \emptyset$ and 
\begin{equation}
  \label{eq:imma}
  \mbox{for general $[E_d] \in \HH^{\circ}$ and all $x,y \in E_d \cap R$ we have $x+y \in \Sym^2(R)^{\circ}$}.
\end{equation}
Since $\dim(\HH^{\circ})=d$ by (i) and the fibers of 
$p_d$ are finite, we have  $\dim(J_d)=d$. Moreover, $J_d$ is reduced, as $\HH^{\circ}$ is.

Consider the locally closed subset $\Z_d$ of elements $x+y \in \Sym^2(R)^{\circ}$ such that the unique lines $\ell_x$ and $\ell_y$ through $x$ and $y$, respectively (which are distinct by definition of $\Sym^2(R)^{\circ}$) intersect.  If through a point of $R$ there passes a unique line, this intersects only finitely many other lines on $R$ (because $R$ is not a cone). Hence either $\Z_d$ is empty or $\dim (\Z_d)=2$. In the former case we have finished. Thus assume that  $\dim (\Z_d)=2$.  We want to prove that the image of $q_d$ does not lie inside $\Z_d$.

We remark that for $x+y \in \im(q_d)$, the fiber $q_d^{-1}(x+y)$ is isomorphic to $\HH^{\circ}(x,y)$, the locus of curves in  $\HH^{\circ}$ passing through $x$ and $y$. The tangent space of the fiber at a  point $[(E_d,x+y)]$ is $H^0(\N_{E_d/X}(-x-y))$. If  $\im(q_d) \subset \Z_d$, we would therefore have:
\begin{equation}
  \label{eq:fibra2P}
  h^0(\N_{E_d/X}(-x-y)) \geq \dim(\HH^{\circ}(x,y)) \geq d-2>0 \; \; \mbox{for all} \; \; [E_d] \in \HH^{\circ}, \; x,y \in E_d \cap R
\end{equation}
(where we have used \eqref{eq:imma}).

Consider now the incidence variety
\[ I_d:=\{ [(E_d,x)] \in \HH^{\circ} \x  R \;|\; x \in E_d \cap R\},\] with projections $f_d: I_d \to \HH^{\circ}$ and $g_d: I_d \to R$.
Since the fibers of 
$f_d$ are finite, we have, as above,  $\dim(I_d)=d$  and $I_d$ is reduced.  By   (iv),  the map $g_d$ is dominant, whence
for general $x \in R$, the fiber $g_d^{-1}(x)$ is reduced and isomorphic to $\HH^{\circ}(x)$, the locus of curves in  $\HH^{\circ}$ passing through $x$, and has dimension $d-2$. The tangent space of the fiber at a general  point $[(E_d,x)]$ is $H^0(\N_{E_d/X}(-x))$, whence, since the general fiber is reduced, we have 
 \begin{equation}
  \label{eq:fibra2P'}
  h^0(\N_{E_d/X}(-x)) =\dim(\HH^{\circ}(x)) = d-2 \; \; \mbox{for general} \; \; [E_d] \in \HH^{\circ}, \; x \in E_d \cap R.
\end{equation} 
Comparing \eqref{eq:fibra2P} and \eqref{eq:fibra2P'}, we see that
\[ H^0(\N_{E_d/X}(-x))=H^0(\N_{E_d/X}(-x-y)) \; \; \mbox{for general} \; \; [E_d] \in \HH^{\circ}, \; x,y \in E_d \cap R,\]
whence, for general $ [E_d] \in \HH^{\circ}$, letting $x_1,\ldots,x_{md}$ denote the intersection points $E_d \cap R$,
we would have
\[h^0(\N_{E_d/X}(-x_i))=h^0(\N_{E_d/X}(-x_1-\cdots-x_{md})>0, \; \; \mbox{for any} \;\; i \in \{1,\ldots,md\}.\]
Therefore, $\N_{E_d/X}$ has a nonzero section vanishing at $md$ points, whence Lemma \ref{lemma:normalb} yields $h^0(\N_{E_d/X}) >d=\dim(\HH')$, contradicting the hypothesis that $\HH'$ is reduced.
\end{proof}

\begin{corollary} \label{cor:g4}
  The spaces $\overline{\HH}_d(X)$ and $\MM_d(X)$ are empty for $d < \lfloor \frac{g+3}{2} \rfloor$.
\end{corollary}

\begin{proof}
  This follows as $2d-g-2 \geq 0$ by Proposition \ref{prop:g4}(i). 
\end{proof}

We finish the section with a result that we will use later:

\begin{lemma} \label{lemma:g34}
Let $d=g \in \{3,4\}$ and $[E_d]$ be  a general point  in a component of $\HH_d(X)$. Then
  $h^1(\N_{E_d/X})=0$.
\end{lemma}

\begin{proof}  
Let $[\F_d] =p([E_d])$ (cf. \eqref{eq:mappap}).   Assume first that there is a  non--zero  section $s' \in H^0(\F_d)$  whose zero scheme $E'$ is  such that the intersection $E' \cap E_d$ is non--empty, transversal and has length $<d$.  A local \label{pageref:loccomp} computation\footnote{Indeed, if $A$ and $B$ are curves in a smooth threefold $X$ intersecting
transversally in one point $p$,  then one locally has $\I_{A/X} |_{B} \cong \O_{B} \+ \O_p$. To prove this, introduce coordinates $(x,y,z)$ around $p$, such that the ideal of $B$ is $(x,y)$ and the ideal of $A$ is $(y,z)$. Then $\I_{A/X}|_{B}$ is locally isomorphic to 
$$
(y,z) \otimes \frac {\mathbb C[[x,y,z]]}{(x,y)} \cong \frac { (y,z)}{(x,y)\cdot (y,z)},
$$
which we have to view as a $\mathbb C[[z]]$--module.  As such we see that $y$ generates its torsion (because $zy=0$), and $z$ generates its free part.}
shows that  $\I_{E'/X} \*_{\O_X} \O_{E_d}$ has torsion along
$Z:=E' \cap E_d$.
Restricting 
\[ \xymatrix{
    0 \ar[r] & \O_X \ar[r]^{s'} & \F_d \ar[r] & \I_{E'/X}(1) \ar[r] & 0} 
\]
to $E_d$ and saturating (and recalling Lemma \ref{lemma:h3h2}(i)) one  obtains
\[
  \xymatrix{
0 \ar[r] & \O_{E_d}(Z)   \ar[r] &  \F|_{E_d}\cong \N_{E_d/X} \ar[r] & \O_{E_d}(Z') \ar[r] & 0}
  \]
 with $\deg(Z')=d-\deg(Z)>0$. Therefore,
  $h^1(\N_{E_d/X})=0$ in this case. 

Let $s_d\in H^0(\F_d)$ be a section vanishing on $E_d$. 
Consider now two general two--dimensional subspaces $V,V' \subset H^0(\F_d)$, such that $V$ contains $s_d$ and $V'$ does  not. (Such subspaces exist because $h^0(\F_d)  \geq \chi(\F_d)=3$ by Lemmas \ref{lemma:FM}(i) and \ref {lemma:h3h2}(v)). Then $V$ and $V'$ define the elements $[S]=w([V])$ and  $[S']=w([V'])$ by \eqref{eq:mappawedge}, and $S$ and $S'$ contain the pencils $\PP(V)$ (containing $E_d$) and $\PP(V')$ whose general members are  smooth elliptic curves of degree $d$. By Proposition \ref{prop:g4}(iii), the pencils have at most one base point, and by the case we treated above,  we may assume that they are base point free. 

Let $C:=S \cap S'$.  We claim that $C$ is reduced. If $g=3$,  we may specialise $V'$ to a general 2--dimensional subspace containing $s_d$; in this case  $S \cap S'=E_3+\ell$, with $\ell$ a line, which is reduced. If $g=4$ and $C$ is not reduced, then the intersection of the two hyperplanes containing $S$ and $S'$ would be a $\PP^3$ in $\PP^5$ that is tangent to $X$ along a component of $C$, against  Zak's theorem on tangencies (see \cite[Thm. 3.2.3]{Rus}).  

Possibly replacing $E_d$ by a general member of $|E_d|$, the intersection $Z:=E_d \cap C$ is transversal and consists of $d$ distinct points, along the smooth locus of $C$.  Pick any $x \in Z$ and pick the unique curve $E'=Z(s')$, for $s' \in H^0(\F)$, in the pencil $\PP(V')$ passing through $x$. Then
$E_d$ and $E'$ can only intersect along $C$, and if the intersection consists of fewer than $d$ points, then the lemma is proved, by what we said above. Hence,
 we may assume that $E \cap E'=Z$. In particular, the surface $\overline{S}:=Z(s \wedge s')$ has $d$ nodes at $Z$. Since $C$ is smooth at $Z$, it cannot lie on $\overline{S}$, as it would be a Cartier divisor on it, because $C \subset S \cap \overline{S}$ and equality holds as both sides have degree $2g-2$. But then
$ 2g-2=C \cdot \overline{S} \geq 2d=2g$,
a contradiction, as desired.
\end{proof}

\section{Proof of Theorem \ref{thm:ellipt}} \label{sec:pf1}

We will use induction on $d$, starting with the case $d=d_g:=\left\lfloor \frac{g+3}{2}\right\rfloor$, recalling that there are no smooth elliptic curves of   degree $d$ in $X$ if $d<d_g$ by Corollary \ref{cor:g4}.

\subsection{The base case $d=d_g$}
The key result is the following:

\begin{prop} \label{prop:contieneE}
  Let  $X \subset \PP^{g+1}$ be a  prime Fano threefold of degree $2g-2$. Then there is a
 codimension--one family of hyperplane sections of $X$ containing smooth elliptic curves of degree $d_g$, whereas a general hyperplane section does not contain such curves.
\end{prop}

\begin{proof}
Let $g \in \{3,\ldots,10,12\}$. By \cite[Thm.\;3.2]{kn} (which is a consequence of the surjectivity of the period map for $K3$ surfaces) there exists a smooth $K3$ surface
$S \subset \PP^g$ containing a  smooth elliptic curve $E$ (whence $E^2=0$) satisfying $\deg(E)=d_g$ and such that $(S,\O_S(1))$ is a so--called {\it $BN$--general polarized $K3$ surface}, as defined by Mukai
in \cite[Def. 3.8]{muk}\footnote{A polarized $K3$ surface $(S,H)$ of genus $g$ is by definition $BN$--general if $h^0(M)h^0(N) \leq g$ for all nontrivial $M,N \in \Pic(S)$ such that $H \sim M+N$. One may prove, using the famous result of Lazarsfeld \cite{laz}, that for $g \in \{3,\ldots,10,12\}$ this condition is equivalent to the fact that all smooth curves $C \in |H|$ are {\it Brill--Noether general}, that is, carry no line bundles $A$ with $\rho(A)<0$, where $\rho$ is the {\it Brill--Noether number}. One may also prove that the general smooth curve in $|H|$ is {\it Brill--Noether--Petri general}, that is, it has injective {\it Petri map} for all line bundles $A$ on $C$. We will however not need this. 
}. From \cite[Thms.\;4.4 and 5.5]{muk} it follows that $S$ can be realized as a hyperplane section of a { \it prime} Fano threefold $X$ of degree $2g-2$ in $\PP^{g+1}$. The general hyperplane section of $X$ does however not contain any smooth elliptic curves, as its Picard group is generated by its hyperplane section  (cf. \cite[Thm.\,3.33]{Vo}).

Consider now
\begin{itemize}
\item the Hilbert scheme $\mathfrak{F}_g$ parametrizing   (smooth)  prime Fano threefolds of degree $2g-2$ in $\PP^{g+1}$,
\item the closure $\mathfrak{K}_g$ of the Hilbert scheme  parametrizing smooth $BN$-general
  $K3$ surfaces of degree $2g-2$ in $\PP^{g+1}$, 
\item the closure $\mathfrak{E}_{g}$ of the Hilbert scheme parametrizing smooth elliptic curves of degree $d_g$ in $\PP^{g+1}$,
  \item the subscheme $\mathfrak{K}^*_g$ of $\mathfrak{K}_g$ parametrizing $[S] \in \mathfrak{K}_g$ such that $S$ contains a member of $\mathfrak{E}_{g}$,
\item the incidence variety $\mathfrak{I}_g$ parametrizing pairs $(S,X)$ such that $[S] \in  \mathfrak{K}_g$, $[X] \in \mathfrak{F}_g$ and $S \subset X$,
\item the natural projections  $\alpha:\mathfrak{I}_g \to \mathfrak{K}_g$  and $\beta:\mathfrak{I}_g \to \mathfrak{F}_g$.

\end{itemize}
It is well-known that both $\mathfrak{K}_g$ and $\mathfrak{F}_g$ are irreducible and that $\mathfrak{K}^*_g$ is a divisor in $\mathfrak{K}_g$. Set 
$k:=\dim (\mathfrak{K}_g)$ and $f:=\dim(\mathfrak{F}_g)$.

For any $[X] \in \mathfrak{F}_g$ the fiber $\beta^{-1}([X])$ is isomorphic to the complete linear system $|\O_X(1)| \cong \PP^{g+1}$. Hence $\dim(\mathfrak{I}_g)=
f+g+1$, whence the general fiber of $p$ has dimension $f+g+1-k$. Thus, special fibers have dimension at least $f+g+1-k$, whence $\dim ( \alpha^{-1}(\mathfrak{K}^*_g)) \geq (k-1)+(f+g+1-k)=f+g$.

By what we said above, $ \alpha^{-1}(\mathfrak{K}^*_g)$ is nonempty and all
fibers of $\beta|_{ \alpha^{-1}(\mathfrak{K}^*_g)}$ have dimension  at most $g$.  
    It follows that
    \[ \dim (\beta( \alpha^{-1}(\mathfrak{K}^*_g))  \geq  \dim ( \alpha^{-1}(\mathfrak{K}^*_g)) -g \geq f+g-g=f.\]
Therefore, $\beta|_{p^{-1}(\mathfrak{K}^*_g)}$ is surjective. 
It follows that {\it any} prime Fano threefold $X \subset \PP^{g+1}$ contains a codimension one family of hyperplane sections containing a curve $E$ that is a member of
$\mathfrak{E}_{g}$, whereas a general hyperplane section does not contain such a curve. 

Let  $(E,S)$ be a general pair consisting of a hyperplane section $S$ of $X$ and a curve $E$ in  $\mathfrak{E}_{g}$ contained in $S$. We may write $E=E'+R$, with $E'$ minimal with respect to the property that $p_a(E')=1$; thus $R$ must consist of smooth rational curves and $E'$ is of degree $d' \leq d_g$. By Lemma \ref{lem:dual}, the surface $S$ has at worst a rational double point (at a general point of $X$). By considering the desingularization of $S$, one sees that $E'$ is linearly equivalent to a smooth, irreducible elliptic curve on $S$. By  Corollary \ref{cor:g4}, we must have $d'=d_g$, so that $E=E'$.  This proves that $S$ contains a smooth elliptic curve of degree $d_g$.
\end{proof}

\begin{remark} A similar reasoning using \cite[Thm. 3.2]{kn} proves that a {\it general}  prime Fano threefold $X \subset \PP^{g+1}$ contains smooth elliptic curves of any degree $d \geq d_g$, but it fails to guarantee that {\it all} such threefolds contain such curves.
\end{remark}

\begin{remark} \label{rem:g=4}  Something can be said about the structure of the 
Hilbert schemes $\overline{\HH}_{d_g}(X)$. 

(i) Let  $X \subset \PP^{4}$ be a smooth quartic threefold. There is a map 
$\overline{\HH}_{3}(X)\to \LL(X)$, sending a cubic $E$ to the unique line residual intersection of $X$ with the plane spanned by  $E$. The fibers are isomorphic to $\PP^2$. 

 (ii) Let  $X \subset \PP^{5}$ be a smooth prime Fano threefold of degree  $6$. Then $X$ is contained in a unique quadric $Q_X$, which is either smooth or a cone with vertex a point, and is cut out by a cubic there. We have $\dim(\overline{\HH}_3(X))=3$ by Propositions \ref{prop:contieneE} and \ref{prop:g4}.  Any cubic $E$ in $\overline{\HH}_3(X)$ spans a plane, which is contained in $Q_X$, and distinct cubics  clearly span  distinct planes. If $Q_X$ is smooth, it contains two families of planes both parametrised by a $\PP^3$, whence $\overline{\HH}_3(X)$ has two irreducible components. If $Q_X$ is singular (that is, $X$ is as in $(\star)_4$), it contains only one family of planes, whence $\overline{\HH}_3(X)$ is irreducible.  

 (iii) Let  $X \subset \PP^{6}$ be a smooth prime Fano threefold of degree  $8$, which is the complete intersection of three quadrics.
 Any smooth elliptic normal quartic $E \subset X$ is a complete intersection  of two quadrics in its span $\langle E \rangle = \mathbb P^3$. Since $X$ is intersection 
of three quadrics in $\mathbb P^6$, then $\langle E \rangle$ is contained in one of the quadrics containing $X$, 
which is therefore singular.  Let $\mathcal H$ be the Hesse curve, namely 
the curve in the net of quadrics $|\mathcal I_{X/ \mathbb P^6}(2)|$ parametrizing singular quadrics. Then any point $h \in \mathcal H$ corresponds to a quadric $Q_h$ having at most a double line as singular locus, 
and it is easy to see that quadrics $Q_h$ having a double line are isolated in $\mathcal H$.  Therefore, the 
maximal dimension for linear spaces contained in $Q_h$, for any $h \in\mathcal H$, is $3$. This gives 
a $2:1$ cover $f : \widetilde{\mathcal H} \longrightarrow \mathcal H$ defined as follows: 
for $h \in \mathcal H$ general, $f^{-1}(h)$ consists of two points 
corresponding to the two distinct rulings of $\mathbb P^3$s contained in $Q_h$. 
Thus $f$ ramifies exactly at points $h \in \mathcal H$ corresponding to quadrics $Q_h$ with a double line. The map $f$ induces a 
map $g : \mathcal P \longrightarrow \widetilde{\mathcal H}$, where 
$\mathcal P$ is a $\mathbb P^3$--bundle over $\widetilde{\mathcal H}$ such that, for 
any $h' \in \widetilde{\mathcal H}$, the fiber $g^{-1}(h')$ parametrizes 
the ruling of $\mathbb P^3$s corresponding to $h'$ and which is contained in $Q_{f(h')}$. Any 
elliptic normal quartic $E \subset X$ determines a point in $\mathcal P$; conversely, any  
point in $\mathcal P$ gives rise to a $\mathbb P^3$ which cuts out on $X$ a normal elliptic 
quartic curve $E$ in $X$. Thus $\overline{\HH}_4(X)$ is isomorphic to $\mathcal P$. 

(iv) For $g\geq 6$ it is shown in  \cite[Thm.\;3.2]{bf} that  $\MM_{d_g}(X)$ is irreducible, whence also  $\overline{\HH}_{d_g}(X)$ is irreducible by Proposition \ref {prop:g4}(i). We will however not use this.

\end{remark}

\begin{lemma}\label{lemma:propF2}
Let  $X \subset \PP^{g+1}$ be a   prime Fano threefold of degree $2g-2$ not as in case $(\star)_4$. Let $[E_{d_g}] \in \HH_{d_g}(X)$ be general in any component of $\HH_{d_g}(X)$. Then $h^1(\N_{E_{d_g}/X})=0$.
     \end{lemma}

     \begin{proof}
We first note that since $d_3=3$, the lemma is proved for $g=3$ by Lemma \ref{lemma:g34}. We may therefore assume that $g \geq 4$.

Let $\HH'$ be an irreducible component of $\HH_{d_g}(X)$ and let $\MM':=p(\HH')$ (cf. \eqref{eq:mappap} and recall that $p$ is surjective by Proposition \ref{prop:g4}(i)). 
Denote by  $\SS'$ the locus in $\SS_{d_g}(X)$ of hyperplane sections of $X$
of the form $Z(s_1 \wedge s_2)$ for (non--proportional) $s_1,s_2 \in H^0(\F_{d_g})$ for some $[\F_{d_g}] \in \MM'$. 
By Proposition \ref{prop:g4}(iii), the general member $S$ of $\SS'$ has at most a rational double point and contains a pencil of curves that are all zero loci of sections of a member $\F_{d_g}$ of $\MM'$ (which we may take to be general), having at most a single base point at the rational double point. We write the pencil as
$|E_{d_g}|$, where we may take $[E_{d_g}] \in \HH'$ to be general. We henceforth write $E=E_{d_g}$ and $\F=\F_{d_g}$ for simplicity.

Assume first that we are in the case where $|E|$ has a base point.
Let $E' \in |E|$ be a smooth curve different from $E$ and 
let $s' \in H^0(\F)$ be such that $E':=Z(s')$. Then $E$ and $E'$ intersect transversally in one point $x$. A local computation as in the footnote of page \pageref{pageref:loccomp} 
shows that  $\I_{E} \*_{\O_X} \O_{E'}$ has torsion along
$x$. Restricting 
\[ \xymatrix{
    0 \ar[r] & \O_X \ar[r]^{s'} & \F \ar[r] & \I_{E'/X}(1) \ar[r] & 0} 
\]
to $E$ and saturating (and recalling Lemma \ref{lemma:h3h2}(i)) one therefore obtains
\[
  \xymatrix{
0 \ar[r] & \O_{E}(x)   \ar[r] &  \F|_{E}\cong \N_{E/X} \ar[r] & \O_{E}(Z) \ar[r] & 0,}
  \]
  with $\deg(Z)=d_g-1>0$. Therefore,
  $h^1(\N_{E/X})=0$, as desired.

  For the rest of the proof we may therefore assume that $|E|$ is base point free on $S$, in particular, we may assume that $E$ does not pass through the possible rational double point of $S$. In particular, the divisor $\Delta:=\O_S(1)(-E)$ is Cartier on $S$. Since $h^0(\I_{E/X}(1)) \geq g+2-d_g \geq 3$ by Lemma \ref{lemma:h3h2}(iii),
the short exact sequence
  \[
 \xymatrix{
 0 \ar[r] & \O_X    \ar[r] &  \O_X(1) \* \I_{E/X} \ar[r] &  \O_S(\Delta) \ar[r] & 0}
     \]
     yields that $h^0(\Delta) \geq 2$, whence $\Delta$ can be represented by an effective nonzero divisor.

\begin{claim}
  $h^1(\N_{E/X})\neq 0$ if and only if there is a section of $\F$ whose zero locus is $\Delta$. This can only occur for $g=4$ or $5$. 
  \end{claim}

\begin{proof}[Proof of claim]
By Lemmas \ref{lemma:h3h2}(v) and \ref{prop:g4}(ii) we have $h^1(\F)=h^2(\F)=0$, whence 
  \[
 \xymatrix{
 0 \ar[r] & \F \* \I_{E/X}   \ar[r] &  \F \ar[r] &  \F|_{E} \cong \N_{E/X} \ar[r] & 0}
     \]
shows that $h^1(\N_{E/X})=h^2(\F \* \I_{E/X})$.
By Lemma \ref{lemma:h3h2}(iv) and 
 \[
\xymatrix{
0 \ar[r] & \F(-1)   \ar[r] &  \F \* \I_{E/X} \ar[r] &  \F|_S(-E) \ar[r] & 0,}
    \] together with Serre duality on $S$ (recalling that $S$, having only rational singularities, is Cohen--Macaulay),    we see that
    \begin{eqnarray*}
      h^2(\F \* \I_{E/X}) & = & h^2(\F|_S(-E))=h^0(\F^*|_S(E)) \\
      & = & h^0(\F|_S(-1)(E))=h^0(\F|_S(-\Delta)).
      \end{eqnarray*}
    By Lemma \ref{lemma:h3h2}(iv) once more and
 \[
\xymatrix{
0 \ar[r] & \F(-1)   \ar[r] &  \F \* \I_{\Delta/X} \ar[r] &  \F|_S(-\Delta) \ar[r] & 0}
    \]
    we get
$h^0(\F|_S(-\Delta))    
  = h^0(\F \* \I_{\Delta/X})$.
  To summarize, we have proved that $h^1(\N_{E/X})=h^0(\F \* \I_{\Delta/X})$.
These are nonzero if and only if
there is a section of $\F$ whose zero locus contains $\Delta$. If this happens, we must have $\deg(\Delta)=2g-2-d_g \leq \deg(E)=d_g$, which can only occur if $g \leq 5$, in which case $\deg(\Delta)=\deg(E)=d_g$, whence the zero locus of the section equals $\Delta$.
\end{proof}

Thus the lemma is proved except for the cases $g=4$ and $5$, which we now treat.

 One  computes $\Delta^2=0$, whence $p_a(\Delta)=1$, and $\deg(\Delta)=d_g$. As in the proof of Proposition \ref{prop:contieneE} we may write $\Delta=\Delta'+R$ with $\Delta'$ minimal with respect to
the property that $p_a(\Delta') = 1$, and one may again check that $\Delta'$ can be represented by a smooth curve. Thus, Corollary \ref{cor:g4} yields that 
$\Delta=\Delta'$; in other words, we may assume that $|\Delta|$ is a pencil whose general member is a smooth elliptic curve of degree $d_g$.

We now treat the cases $g=4$ and $g=5$ separately.

First we examine the case $g=4$, where $d_g=3$. We prove that if $h^1(\N_{E/X})\neq 0$, then $X\subset \PP^5$ is contained in a quadric cone. We argue by contradiction and assume that $X$ is contained in a smooth quadric $Q$. 
Recalling the map \eqref {eq:mappawedge}, we may write $[S]=w([V])$, where
$V$ is a general member of ${\rm Grass}(2,H^0(\F))={\rm Grass}(2,4)$. Then the pencil $|E|$, whose general member is a smooth plane cubic, equals $\PP(V)$. The surface $S$ is contained in a hyperplane $\pi_S$ of $\PP^5$ and it is there a complete intersection of type $(2,3)$. The unique quadric $Q_S$ containing $S$ is the intersection $Q_S=Q\cap \pi_S$. Since $S$ contains a pencil  of plane cubics, the quadric $Q_S$ contains a pencil of planes, thus it is singular. Since $Q$ is smooth, $Q_S$ must be of rank $4$, hence it is a cone with vertex a point, and therefore it contains a unique other pencil of planes that cut out on $S$ the pencil $|\Delta|$,  with $|E|$ and $|\Delta|$ distinct.  By the  claim, there is another $V'$ in ${\rm Grass}(2,H^0(\F))$ such that $w([V'])=S$ and $\PP(V')=|\Delta|$, and then $w^{-1}([S])=\{[V],[V']\}$. Let us now interpret the map $w$ in \eqref {eq:mappawedge} as a generically $2:1$ map
$$
w: {\rm Grass}(2,H^0(\F))\longrightarrow (\PP^5)^*
$$
We have that $w({\rm Grass}(2,H^0(\F)))= Q^*$. The above argument shows that $w$ cannot have any ramification point, since it would correspond to a hyperplane section of $Q$ of rank 3. This is absurd since $Q^*$, which is itself a smooth quadric in $\PP^5$, is simply connected. This gives a contradiction  that proves that $Q$ must be a quadric cone.

Next we consider the case $g=5$, where $ d_g  =4$. We assume that $h^1(\N_{E/X})\neq 0$ and want to reach a contradiction. As we saw above, the surface $S$   is endowed with two distinct pencils $|E|$ and $|\Delta|$
of elliptic quartic curves. Then $E\cdot \Delta=4$ and, by the claim, $E$ and $\Delta$ are both zero schemes of two sections $s$ and $s'$ of $\F$.  We can consider $W=\langle s,s'\rangle$ and $[S']=w([W])$; then $S'$ is a surface with 4 singular points at the intersection of $E$ and $\Delta$, which we can assume to be general in $|E|$ and $|\Delta|$. The four singular points in question, forming  a scheme $Z$,  lie in a plane $\alpha$, because on both $E$ and $\Delta$ they are a hyperplane section. Let us consider the projection $\pi$ of  $S'$ from $\alpha$ to a plane $\beta$. The surface $S'$ has a pencil $\PP(W)$ of elliptic quartic curves that contains both $E$ and $\Delta$. The curves of this pencil, all containing the scheme $Z$, are contracted to points by the projection $\pi$ and therefore the image of $S'$ under this projection is a curve. A general hyperplane section of $S'$ containing $E$ consists of $E$ plus another elliptic quartic curve $E'$ that, as $E$, passes through $Z$. Hence $E'$ is also contracted by $\pi$, thus $E'$ belongs to $\PP(W)$ as well as $E$. This implies that $\O_{S'}(1) \sim 2E \sim 2E'$, 
whence $\pi(S')$ is a conic. This proves that $S'$ sits on a quadric $Q'$ that is a cone with vertex the plane $\alpha$. The quadric  $Q'$ is the hyperplane section of a quadric $Q$ in $\PP^6$ containing $X$, and $Q$ must be singular with vertex a line. 

 As in the proof of Proposition \ref{prop:g4}\footnote{Again we pretend
  that $\MM_4(X)$ is a fine moduli space; if not we argue similarly at a local level.} let us consider the universal grassmannian  bundle $\GG\to \MM_4(X)$, whose points correspond to pairs $([\F], [V])$ with $[\F] \in  \MM_4(X)$ and $[V] \in {\rm Grass}(2,H^0(\F))$, and the map 
$
\mathfrak w: \GG\to |\O_X(1)|
$
as in \eqref{eq:mappawedge-univ}.
In the proof of Proposition \ref{prop:g4} we proved that $\dim(\GG)=g=5$ and that $\mathfrak{w}$ is generically finite. 

As in the case $g=4$, we have an involution $\iota: \mathbb G\dasharrow \mathbb G$, such that for a general $([\F], [V])\in \mathbb G$, one has $\iota([\F], [V])=([\F], [V'])$ and $\mathfrak w([\F], [V])=\mathfrak w ([\F], [V'])$, with $\PP(V)$ and $\PP(V')$ two pencils of elliptic quartic curves on the  surface defined by $\mathfrak w([\F], [V])=\mathfrak w ([\F], [V'])$. The fixed points of the involution $\iota$ form a non--zero divisor $D$ in $\mathbb G$. The image via the map $\mathfrak w$ of a generic point $([\F], [V])$ in a suitable component of $D$ is a surface $S$ with four singular points on a plane, with the unique pencil $\PP(V)$ of elliptic quartic curves passing through the four singular points. Such surfaces therefore lie, as we saw, on a quadric cone with vertex a plane. Accordingly, we have a family of quadrics containing $X$ that are cones with vertex a line. This family must have at least dimension 1, because a unique quadric in $\PP^6$ with vertex a line  gives rise only to a 3--dimensional family of hyperplane sections with vertex a plane. In conclusion, in the net $\mathcal N$ of quadrics that cut out $X$, there is at least a one--dimensional subfamily 
$\mathcal V$ of quadrics with vertex a line. Let $[Q] \in \mathcal V$ be a general point. Then  
$\mathcal N$ is spanned by the tangent line $L$ to $\mathcal V$ at $Q$ and by another general quadric $Q'$. Now the quadrics in the pencil $L$ all contain the singular line $\ell$ of $Q$, so they cut out a complete intersection singular along $\ell$. By intersecting with $Q'$ we see that $X$ must be singular at the intersection points of $Q'$ with $\ell$, a contradiction. 
\end{proof}

     \begin{lemma}\label{lem:capocc} Let  $X \subset \PP^{5}$ be a smooth prime Fano threefold of degree $6$  as in case $(\star)_4$. Then $\HH_3(X)$ is nonreduced with $4$--dimensional tangent space at every point. 
     \end{lemma}

     \begin{proof} By Lemma \ref{lemma:normalb} and Remark \ref{rem:g=4}(ii) it suffices to prove that $h^1(\N_{E/X}) >0$ for general $[E] \in \HH_3(X)$.  By definition, $X$ is contained in a singular quadric $Q_X$ with vertex one point. From the exact sequence
$$
\xymatrix{
0\ar[r] & \N_{E/X}\ar[r] &\N_{E/Q_X}\ar[r] & \N_{{X/Q_X}|_E}\cong \O_E(3)\ar[r]& 0
}
$$we see that it suffices to prove that $h^1(\N_{E/Q_X})>0$. 

Consider the incidence variety
$$
I=\{(Q,C):  Q\in |\O_{\PP^5}(2)|,\,\,  C\subset Q,\,\, C \,\, \text{plane cubic}\}
$$
with the  projection $f: I\to |\O_{\PP^5}(2)|$. If $Q\in |\O_{\PP^5}(2)|$ is smooth, the fiber $f^{-1}(Q)$ consists of two irreducible components of dimension 12, and both are $\PP^9$--bundles over $\PP^3$. Indeed, a plane cubic $C$ sits in $Q$ if and only if the plane its spans lies in $Q$, and $Q$ contains two families of planes both parametrised by a $\PP^3$. 
If $Q$ is singular with vertex one point, then the fiber $f^{-1}(Q)$ consists of only one irreducible component of dimension 12, which is again a $\PP^9$--bundle over $\PP^3$, since $Q$ contains only one family of planes. Therefore $f^{-1}(Q_X)$ is non--reduced. Since it is isomorphic to the component of the Hilbert scheme of $Q_X$ parametrising plane cubics, we see that this component is non--reduced. Hence $h^1(\N_{E/Q_X})>0$, as wanted. 
\end{proof}

Proposition \ref{prop:contieneE} guarantees that $\HH_{d_g}(X) \neq \emptyset$ on any prime Fano threefold $X \subset \PP^{g+1}$, and Proposition \ref{prop:g4}(i) yields that each component has dimension $d_g$.
By Lemma \ref{lemma:propF2} and standard theory of Hilbert schemes, $\HH_{d_g}(X)$ is smooth at a general point of each component, whence reduced, except in case $(\star)_4$, where $\HH_3(X)$ is irreducible of dimension $3$ by Remark \ref{rem:g=4}(ii), with $4$--dimensional tangent space everywhere by
Lemma \ref{lem:capocc}. 

Finally, for general $[E_{d_g}] \in \HH_{d_g}(X)$ one has $h^1(\I_{E_{d_g}(X)/X}(1))=0$ by  Proposition \ref{prop:g4}(ii), whence $E_{d_g}(X)$ is an elliptic normal curve (cf. Remark \ref{rem:elln}).

    This concludes the proof of Theorem \ref{thm:ellipt} in the case $d=d_g$.

\subsection{The inductive step} Assume that $d_g \leq d \leq g+2$ and that
Theorem \ref{thm:ellipt} has been proved for the integer $d$. 
We will prove that it also holds for $d+1$, that is, that $\HH_{d+1}(X)$ has a reduced component of dimension $d$ containing points representing elliptic normal curves if $d+1 \leq g+2$ and nondegenerate curves if $d+1=g+3$. To do so, it will be sufficient, by standard theory of Hilbert schemes and Remark \ref{rem:elln}, to prove the existence of a smooth elliptic curve $E_{d+1}$ of degree $d+1$ in $X$ such that $h^1(\N_{E_{d+1}/X})=0$ and 
$$h^1(\I_{E_{d+1}/X}(1))=\begin{cases}
0, & \mbox{if} \; d+1 \leq g+2, \\
1, & \mbox{if} \; d+1=g+3.
\end{cases}
$$ 

Denote by $\HH_d(X)'$ a component of $\HH_d(X)$ satisfying the conditions of Theorem \ref{thm:ellipt}. This means that $\dim(\HH_d(X)')=d$ and, except for the case $(\star)_4$ with $d=3$, the scheme $\HH_d(X)'$ is reduced, whence $h^0(\N_{E_d/X})=d$ for general $[E_d] \in \HH_d(X)'$, equivalently $h^1(\N_{E_d/X})=0$ by Lemma \ref{lemma:normalb}.  Furthermore, $h^1(\I_{E_{d}/X}(1))=0$ (cf. Remark \ref{rem:elln}).

{\bf $\bullet$ The case $d \leq g+1$.} By Proposition \ref{prop:g4}(v), the general curve $E_d$ in $\HH_d(X)'$ intersects $R_1(X)$ transversally in distinct points, and through each point there passes a (unique) line intersecting $E_d$ transversally in one point. Pick any of these lines $\ell$ and call the intersection point $P$. Set $E'=E_d \cup\ell$. Recall that the first cotangent sheaf $T^1_{E'}$ of $E'$  is isomorphic to $\O_P$, since $E'$ has the node at $P$ as its only singularity.   By local computations (see, e.g., the proof of \cite[Thm. 4.1]{HH}), we have 
the exact sequences
\begin{eqnarray}
  \label{eq:sp1}
& \xymatrix{  0 \ar[r] & \N_{E'/X} \ar[r] & \N_{E'/X}|_{E_d} \+ \N_{E'/X}|_{\ell} \ar[r] & \N_{E'/X}|_P \cong \O_P \+ \O_P \ar[r] &  0,} \\
\label{eq:sp2}  & \xymatrix{0 \ar[r] & \N_{\ell/X} \ar[r] & \N_{E'/X}|_{\ell} \ar[r] & \O_P \ar[r] &  0,} \\
  \label{eq:sp3}  & \xymatrix{0 \ar[r] & \N_{E_d/X} \ar[r] & \ar[r] \N_{E'/X}|_{E_d} & \O_P \ar[r] &  0.}
\end{eqnarray}
Thus, from \eqref{eq:sp1}--\eqref{eq:sp3}, and recalling \eqref{eq:lines} and Lemma \ref{lemma:h3h2}(i), we find
\begin{eqnarray*}
  \chi(\N_{E'/X}) & = & \chi(\N_{E'/X}|_{E_d})+ \chi(\N_{E'/X}|_{\ell})-2 \\
  & = & \left(\chi(\N_{E_d/X})+1\right)+\left(\chi(\N_{\ell/X})+1\right)-2 \\
  & = & \chi(\N_{E_d/X})+\chi(\N_{\ell/X})=d+1.
\end{eqnarray*}
Hence, by standard deformation theory, the Hilbert scheme of $X$ has dimension
at least $d+1$ in a neighborhood $\UU$ of $[E']$. Since the sublocus of reducible curves of the form $E_d\cup\ell$ has dimension equal to $\dim(\HH_d)=d$ (by Proposition \ref{prop:g4}(i)), the general member of $\UU$ must be a smooth elliptic curve. Thus, we have proved that $\HH_{d+1}(X)$ is nonempty and has a component $\HH_{d+1}(X)'$ containing elements of the form $E_d\cup\ell$ in its closure, with $[E_d] \in \HH_d(X)'$ general and $\ell$ a line intersecting $E_d$ transversally in one point. 

We now  prove that $\HH_{d+1}(X)'$  contains elements $E_{d+1}$ such that $h^1(\N_{E_{d+1}/X})=0$. This is automatically satisfied if $d+1=g=4$ by Lemma \ref{lemma:g34}. Therefore, we may 
assume that $h^0(\N_{E_d/X})=d$ and $h^1(\N_{E_{d}/X})=0$ for general $[E_{d}] \in \HH'_{d}$.

As we just saw, we may specialize a general $E_{d+1}$ in  $\HH_{d+1}(X)'$ in a flat family to a curve $E':=E_{d}\cup\ell$, with $E_d$ a general member of $\HH_d(X)'$ and $\ell$ a line intersecting $E_d$ transversally in one point $P$. 
Arguing by contradiction, using Lemma \ref{lemma:normalb}, we will assume that $\N_{E_{d+1}/X} \cong \O_{E_{d+1}} \+ \O_{E_{d+1}}(1)$. In particular, $h^0(\N_{E_{d+1}/X})=d+2$, whence $H^0(\N_{E'/X})$ contains a $(d+2)$--dimensional space $V$ of limit sections. Consider $V_{\ell}$ (respectively $V_{E_d}$) the subspace of $H^0(\N_{E'/X})$ of sections identically zero on $\ell$ (resp. $E_d$). Then, by \eqref{eq:sp1} and \eqref{eq:sp3}, we find
\[ \dim (V_{\ell}) \leq h^0(\N_{E'/X}|_{E_d}) \leq h^0(\N_{E_d/X})+h^0(\O_P) =d+1.\]
Similarly, from \eqref{eq:sp1} and \eqref{eq:sp2}, together with \eqref{eq:lines}, we find
\[ \dim (V_{E_d}) \leq h^0(\N_{E'/X}|_{\ell}) \leq h^0(\N_{\ell/X})+h^0(\O_P)  \leq 2+1=3<d+1.\]
Thus, both $V_{\ell}$ and $V_{E_d}$ are properly contained in $V$. This means that the limit of a general section $s: \O_{E_{d+1}} \to \N_{E_{d+1}/X}$ is a section $s':\O_{E'} \to \N_{E'/X}$ not vanishing along any component of $E'$, that is, being still injective. Since $s$ is split, the same is true for $s'$, that is, $\N_{E'/X} \cong \O_{E'} \+ \coker(s')$. Since $\N_{E'/X}$ is locally free, as $E'$ is nodal, whence a local complete intersection, $\coker(s')$ must be locally free, hence isomorphic to $\O_{E'}(1)$, as $\det(\N_{E'/X}) \cong \O_{E'}(1)$. We have therefore proved that $\N_{E'/X} \cong \O_{E'} \+ \O_{E'}(1)$. Restricting to $\ell$ we obtain $\N_{E'/X}|_{\ell} \cong \O_{\ell} \+ \O_{\ell}(1)$. Hence
\eqref{eq:sp2} yields $\chi(\N_{\ell/X})=2$, in contradiction with \eqref{eq:lines}. We have therefore proved that the general member $E_{d+1}$ in  $\HH_{d+1}(X)'$ satisfies $h^1(\N_{E_{d+1}/X})=0$.

 We have left to prove that $h^1(\I_{E_{d+1}/X}(1))=0$. This is satisfied if $d \leq g$ by Proposition \ref{prop:g4}(ii). 
If $d=g+1$ we may as above specialize $E_{g+2}$ in a flat family to $E':=E_{g+1}\cup\ell$. By what we just proved, $h^1(\I_{E_{g+1}/X}(1))=0$, whence
$h^0(\I_{E_{g+1}/X}(1))=1$ by Lemma \ref{lemma:h3h2}(iii). Since, by Proposition \ref{prop:g4}(v), there is a hyperplane section of $X$ containing $E_{g+1}$ but not $\ell$, we have $h^0(\I_{E_{g+1}\cup\ell/X}(1))=0$. Hence, by semicontinuity, $h^0(\I_{E_{g+2}/X}(1))=0$ as well, so that $h^1(\I_{E_{g+2}/X}(1))=0$ by Lemma \ref{lemma:h3h2}(iii) again, as desired.

This concludes the proof of Theorem \ref{thm:ellipt} in the case $d \leq g+1$.

{\bf $\bullet$ The case $d = g+2$.} 
By Proposition \ref{prop:g4}(v)--(vi), the general curve $E_{g+1}$ in $\HH_{g+1}(X)'$ intersects $R_1(X)$ in at least two points $x$ and $y$ through which there pass unique lines $\ell_x$ and $\ell_y$ and such that $\ell_x \cap \ell_y = \emptyset$ and each intersects $E_{g+1}$ transversally at one point. By the first part of the proof, $E_{g+1} \cup\ell_x$ deforms to a smooth elliptic curve
$E_{g+2}$ in $\HH_{g+2}(X)$ having the property that $h^1(\N_{E_{g+2}/X})=0$. Thus, we find a component $\HH_{g+2}(X)'$ of $\HH_{g+2}(X)$ of dimension $\dim(\HH_{g+2}(X)')=g+2$ with general member $E_{g+2}$ satisfying $h^1(\N_{E_{g+2}/X})=0$ and containing
$E_{g+1} \cup\ell_x$ in its closure. The latter curve
has the property that the line $\ell_y$ intersects it transversally in one point. Since this is an open property among members of  the closure of $\HH_{g+2}(X)'$,
and the general member intersects finitely many lines,
it also holds for the general one. This means that for the general $[E_{g+2}]\in \HH_{g+2}(X)'$, there is a line $\ell$ intersecting $E_{g+2}$ transversally in one point. The first part of the proof applied once more yields that $E_{g+2}\cup\ell$ deforms to a smooth elliptic curve
$E_{g+3}$ in $\HH_{g+3}(X)$ having the property that $h^1(\N_{E_{g+3}/X})=0$.

Since $h^0(\I_{E_{g+2}/X}(1))=h^1(\I_{E_{g+2}/X}(1))=0$ by the induction hypothesis and Lemma \ref{lemma:h3h2}(iii), we also have $h^0(\I_{E_{g+2} \cup \ell/X}(1))=0$, whence $h^0(\I_{E_{g+3}/X}(1))=0$ by semicontinuity. By Lemma \ref{lemma:h3h2}(iii) again, we find $h^1(\I_{E_{g+3}/X}(1))=1$.

This concludes the proof of Theorem \ref{thm:ellipt}.

\section{Proof of Theorem \ref{thm:prFanoACM}} \label{sec:pf2}

Let $\HH_d(X)'$ be a component of $\HH_d(X)$ satisfying the conditions of Theorem \ref{thm:ellipt}. Let $\MM_d(X)'$ be a component of $\MM_d(X)$ containing $p(\HH_d(X)')$, cf. \eqref{eq:mappap}. We will prove $\MM_d(X)'$ and its general member satisfy the conditions of Theorem \ref{thm:prFanoACM}.

Let $[E_d] \in \HH_d(X)'$ be general and $[\F_d]=p([E_d])$ (cf. \eqref{eq:mappap}). Then $\F_d$ satisfies \eqref{eq:propFi} by
\eqref{eq:propF} and is slope--stable.

In the case where $d=3$ and $X$ is as in $(\star)_4$, we have that $\overline{\HH}_3(X)$ is irreducible by Remark \ref{rem:g=4}(ii), whence $\MM_3(X)=p(\overline{\HH}_3(X))$ is also irreducible, and of dimension $0$, 
by Proposition \ref{prop:g4}(i). Moreover, by Lemmas \ref{lemma:normalb} and \ref{lemma:propF2}, we have
$h^1(\N_{E_3/X})=1$ in this case; therefore, as
$h^1(\F_3)=0$ by Proposition \ref{prop:g4}(ii), we have $h^2(\F_{3} \* \F_{3}^*) =1$ by Lemma \ref{lemma:h3h2}(vi). Using the remaining identities in Lemmas \ref{lemma:FM}(ii) and  \ref{lemma:h3h2}(vi) and standard theory on moduli spaces of vector bundles, we find that the tangent space at the general point of $\MM_3(X)'$ has dimension
\[ h^1(\F_3 \* \F_3^*)=-\chi(\F_3 \* \F_3^*)
+ h^0(\F_3 \* \F_3^*)+h^2(\F_3 \* \F_3^*)-
  h^3(\F_3 \* \F_3^*)=1.\]

  In the remaining cases,
$\HH_d(X)'$ is reduced of dimension $d$, whence $h^0(\N_{E_d/X})=d$ and $h^1(\N_{E_d/X})=0$ for general $[E_d] \in \HH_d(X)'$ (cf. Lemma \ref{lemma:normalb}). Hence $h^2(\F_d \* \F_d^*)=h^3(\F_d \* \F_d^*)=0$ by Lemma \ref{lemma:h3h2}(vi), in addition to $h^0(\F_d \* \F_d^*)=1$ and $\chi(\F_d \* \F_d^*)=g-2d+3$ by Lemma \ref{lemma:FM}(ii). 
By standard theory  of moduli spaces  of vector bundles (see, e.g., \cite[Prop. 2.10]{ch}),
$\MM_d(X)'$ is smooth at $[\F_d]$
of dimension
\[ h^1(\F_{d} \* \F_{d}^*)=-\chi(\F_{d} \* \F_{d}^*)+h^0(\F_{d} \* \F_{d}^*)= 2d-g-2.\]

We have left to prove that the general member $\F'_d$ in $\MM_d(X)'$ is $ACM$. Since $h^2(\F'_d(n))=h^1({\F'}_d^*(-n-1))=h^1(\F'_d(-n))$ for all $n \in \ZZ$, we see that it suffices to prove that  $h^1(\F'_d(n))=0$ for all $n \in \ZZ$.

From \eqref{eq:HS1} and \eqref{eq:IS1} we see that
\[ h^1(\F_d(n))=h^1(\I_{E_d/X}(n+1)) = \cork\left\{H^0(\O_X(n+1) \to H^0(\O_{E_d}(n+1)\right\}\]
for all $n \in \ZZ$. Thus, as $X$ is projectively normal, $\F_d$ is $ACM$ if and only if $E_d$ is projectively normal, which happens precisely when $E_d$
is an elliptic normal curve. We are therefore done by Theorem \ref{thm:ellipt} in the case $d \leq g+2$.

In the case $d=g+3$ we have 
$h^1(\I_{E_{g+3}/X}(1))=1$ by Theorem \ref{thm:ellipt} and Remark \ref{rem:elln}.
Consequently, $h^0(\F_{g+3})=h^1(\F_{g+3})=1$ and $h^2(\F_{g+3})=h^3(\F_{g+3})=0$
by Lemmas \ref{lemma:FM}(i) and \ref{lemma:h3h2}(v), so that $\F_{g+3}$ is not $ACM$.
On the other hand we have:

\begin{lemma} \label{lemma:Ur2}
  For general $[\F'_{g+3}] \in \MM_{g+3}(X)'$ we have $h^i(\F'_{g+3})=0$ for all $i$.
In particular, $\F'_{g+3}(1)$ is Ulrich and $\F'_{g+3}$ is $ACM$. 
\end{lemma}

\begin{proof}
  We have $\dim(\HH_{g+3}(X)')=g+3$ by Theorem \ref{thm:ellipt}, and $\dim(\MM_{g+3}(X)')=2(g+3)-g-2=g+4$ by what
we  just proved. Thus  $p|_{\HH_{g+3}(X)'}: \HH_{g+3}(X)' \to \MM_{g+3}(X)'$ is not surjective, whence
  $h^0(\F'_{g+3})=0$ for general $[\F'_{g+3}] \in \MM_{g+3}(X)'$. By semicontinuity,  $h^2(\F'_{g+3})=h^3(\F'_{g+3})=0$. Lemma \ref{lemma:FM}(i) then yields $h^1(\F'_{g+3})=0$. This proves the first assertion.

By Serre duality and the fact that ${\F'_{g+3}}^* \cong \F'_{g+3}(-1)$, we also have
\[ h^i(\F'_{g+3}(-2))=h^{3-i}({\F'_{g+3}}^*(1))=h^{3-i}(\F'_{g+3})=0 \;\;
\mbox{for all} \;\;  i.\] 
    By Lemma \ref{lemma:h3h2}(iv) and semicontinuity we have   $h^i(\F'_{g+3}(-1))=0$ for all $i$. Therefore, $\F'_{g+3}(1)$ is Ulrich, whence in particular all its twists are $ACM$ (see, e.g, \cite[(3.1)]{be2}). 
\end{proof}

This concludes the proof of Theorem \ref{thm:prFanoACM}.

Note that we have also proved the following

\begin{corollary} \label{cor:exind1rk2}
  Let $X \subset \PP^{g-1}$ be a prime Fano threefold of degree $2g-2$. There exists a slope--stable rank--two Ulrich bundle with determinant $\O_X(3)$ on $X$. The moduli space of such bundles has a reduced component of dimension $g+4$ and its general member $\E$ satisfies
  \begin{equation}
    \label{eq:comoFF}
   h^0(\E \* \E^*)=1, \; h^1(\E \* \E^*)=g+4, \; h^2(\E \* \E^*)=h^3(\E \* \E^*)=0.
  \end{equation}
 \end{corollary}

 \begin{proof} This follows from the case $d=g+3$ of Theorem \ref{thm:prFanoACM} and Lemma \ref{lemma:Ur2}, letting $\E:=\F_{g+3}(1)$. Recall that we have proved that $h^2(\F_{g+3} \* \F_{g+3}^*)=0$; the rest of \eqref{eq:comoFF} follows from Lemmas \ref{lemma:FM}(ii) and \ref{lemma:h3h2}(vii) and semicontinuity. 
 \end{proof}

\section{Ulrich bundles of higher ranks on prime Fano threefolds} \label{sec:pf3}

In this section we will inductively define families of vector bundles of all even ranks on any prime Fano threefold, whose general members will  be slope--stable Ulrich bundles, and thus prove Theorem \ref{thm:prFano}.

We start by defining the scheme $\UU_X(1)$ to be any irreducible component of the moduli space of vector bundles on $X$ containing rank--two slope--stable Ulrich bundles $\E$ with
$\det(\E)=\O_X(3)$ and satisfying \eqref{eq:comoFF}, which is nonempty by  Corollary \ref{cor:exind1rk2}.  

Having defined $\UU_X(h)$  for some $h \geq 1$, we define $\UU_X(h+1)$ to be the (a priori possibly empty)   component of the moduli space of Ulrich bundles on $X$ 
containing bundles $\F_{h+1}$ that are  {\it non--split}   extensions of the form
\begin{equation} \label{eq:estensione}
  \xymatrix{
0 \ar[r] & \E'_1 \ar[r] & \F_{h+1} \ar[r] & \E_{h} \ar[r] & 0, 
  }
\end{equation}
with $[\E'_1] \in \UU_X(1)$  and $[\E_{h}] \in \UU_X(h)$, and such that $\E'_1 \not \cong \E_{h}$ when $h=1$.  We let $\UU_X(h+1)^{\rm ext}$ denote the locus in $\UU_X(h+1)$ of bundles that are  non--split   extensions of the form \eqref{eq:estensione}. 

In the next lemma and its proof we will use the following notation: $\E'_1$ will be a general member of $\UU_X(1)$ and $\E_h$ will be a general member of $\UU_X(h)$, with $\E_h \not \cong \E'_1$  if $h=1$.  We will denote by $\F_{h}$ a general member of  $\UU_X(h)^{\rm ext}$. 
In bounding cohomologies, we will use the fact that $\E_{h}$ {\it specializes} to $\F_{h}$  in a flat family.

\begin{lemma} \label{lemma:indu}
Let $h \in \ZZ^+$ and  assume  $\UU_X(k) \neq \emptyset$ for all $1 \leq k \leq h$. Then
 \begin{itemize}
 \item[(i)] $h^j(\E_h \* {\E'_1}^*)=h^j(\E'_1 \* \E_h^*)=0$ for $j=2,3$,
  \item[(ii)] $\chi(\E_h \* {\E'_1}^*)=\chi(\E'_1 \* \E_h^*)=-h(g+3)$,
    \item[(iii)] $h^j(\E_h \* \E_h^*)=0$ for $j=2,3$,
 \item[(iv)] $\chi(\E_h \* \E_h^*)=-h^2(g+3)$,

\end{itemize}
\end{lemma}

\begin{proof} For $h=1$, (iii) and (iv)  follow from \eqref{eq:comoFF} and semicontinuity. As for 
(i), the vanishings hold  when $\E_1' = \E_1$ by \eqref{eq:comoFF}, and thus, by semicontinuity, they also hold for a general pair 
  $([\E_1'], [\E_1]) \in \UU_X(1) \times \UU_X(1)$. Similarly, (ii)  follows from \eqref{eq:comoFF}, since the given $\chi$ is constant as $\E_1$ and $\E'_1$ vary in $\UU_X(1)$.

We now prove the statements for any integer $h \geq 2$ by induction. Assume therefore that they are  satisfied for all integers less than $h$. 

  (i) Let $j \in \{2,3\}$. By specialization and \eqref{eq:estensione} we have
  \[ h^j(\E_h \* {\E'_1}^*) \leq h^j(\F_h\* {\E'_1}^*) \leq h^j(\E'_1\* {\E'_1}^*) + h^j(\E_{h-1}\* {\E'_1}^*),\]
  and this is $0$ by induction. Similarly, by specialization and the dual of \eqref{eq:estensione} we have
  \[ h^j(\E'_1 \* {\E_h}^*) \leq h^j(\E'_1 \* {\F_h}^*) \leq h^j(\E'_1\* {\E'_1}^*) + h^j(\E'_1\* {\E_{h-1}}^*),\]
  which is again $0$ by induction.

(ii) By specialization, \eqref{eq:estensione} and induction we have
\begin{eqnarray*}
  \chi(\E_h \* {\E'_1}^*) & = & \chi(\F_h\* {\E'_1}^*) = \chi(\E'_1\* {\E'_1}^*) + \chi(\E_{h-1}\* {\E'_1}^*) \\
  & =  & -(g+3)-(h-1)(g+3)=-h(g+3).
  \end{eqnarray*}
Likewise, by specialization, the dual of \eqref{eq:estensione} and induction we have
\begin{eqnarray*}
  \chi(\E'_1 \* {\E_h}^*) & = & \chi(\E'_1 \* {\F_h}^*) = \chi(\E'_1\* {\E'_1}^*) + \chi(\E'_1\* {\E_{h-1}}^*) \\
  & =  & -(g+3)-(h-1)(g+3)=-h(g+3).
\end{eqnarray*}
  
  (iii) Let $j \in \{2,3\}$. By specialization, \eqref{eq:estensione} and its dual we have
  \begin{eqnarray*}
    h^j(\E_h \* \E_h^*) & \leq & h^j(\F_h \* \F_h^*) \leq
                                 h^j(\E'_1 \* \F_h^*)+h^j(\E_{h-1} \* \F_h^*) \\
                        & \leq & h^j(\E'_1 \* {\E'_1}^*)+h^j(\E'_1 \* \E_{h-1}^*)+
                                 h^j(\E_{h-1} \* {\E'_1}^*) + h^j(\E_{h-1} \* \E_{h-1}^*),
    \end{eqnarray*}
which is $0$ by induction. 

(iv) By specialization, \eqref{eq:estensione} and its dual we have
  \begin{eqnarray*}
    \chi(\E_h \* \E_h^*) & = & \chi(\F_h \* \F_h^*) =
                                 \chi(\E'_1 \* \F_h^*)+\chi(\E_{h-1} \* \F_h^*) \\
                   & = & \chi(\E'_1 \* {\E'_1}^*)+\chi(\E'_1 \* \E_{h-1}^*)+ \chi(\E_{h-1} \* {\E'_1}^*) + \chi(\E_{h-1} \* \E_{h-1}^*).
    \end{eqnarray*}
    By induction, this equals
    \[
            -(g+3)-(h-1)(g+3)-(h-1)(g+3)-(h-1)^2(g+3)=-h^2(g+3).
      \]

\end{proof}

\begin{proposition} \label{prop:h+1}
  For all $h \in \ZZ^+$ the scheme $\UU_X(h)$ is nonempty
and its general member $\E$ is Ulrich and satisfies $\rk(\E)=2h$, $\det(\E)=\O_X(3h)$ and $h^j(\E \* \E^*)=0$ for $j=2,3$. 
\end{proposition}

\begin{proof}
 We prove this by induction on $h$, the case $h=1$ being satisfied by the choice of $\UU_X(1)$.  By Lemma \ref{lemma:indu}(i)--(ii) we have, for general $[\E_{h}] \in
  \UU_X(h)$ and $[\E'_1] \in
  \UU_X(1)$, that  
  \begin{equation} \label{eq:dimext}
    \dim ({\rm Ext}^1(\E_{h},\E'_1))  =  h^1(\E'_1 \* \E_{h}^*) =
    -\chi(\E'_1 \* \E_{h}^*)+h^0(\E'_1 \* \E_{h}^*) \geq h(g+3)>0.
    \end{equation}
    Hence, by definition, $\UU_X(h+1) \neq \emptyset$. Its members have ranks $2(h+1)$ and determinant $\O_X(3(h+1))$. It is  immediate  that extensions of Ulrich bundles are still Ulrich bundles, so the general member $\E_{h+1}$ of 
$\UU_X(h+1)$ is an Ulrich bundle. It satisfies 
$h^j(\E_{h+1} \* \E_{h+1}^*)=0$ for $j=2,3$ by Lemma \ref{lemma:indu}(iii).
\end{proof}

To finish the proof of 
Theorem \ref{thm:prFano} we have left 
to prove that the general member of $\UU_X(h)$ is slope--stable and that $\UU_X(h)$ is smooth of dimension $h^2(g+3)+1$ at its general point. We will again prove this by induction on $h$. 
First we need an auxiliary result.

\begin{lemma} \label{lemma:uniquedest1}
  Let $\F_{h+1}$ be a general member of  $\UU_X(h+1)^{\rm ext}$, sitting in an extension like \eqref{eq:estensione}.
Assume furthermore that $\E'_1$ and $\E_h$ are slope--stable.

  Let $\G$ be a destabilizing subsheaf of $\F_{h+1}$. Then $\G^{*} \cong  {\E_1'}^*$ and $\left(\F_{h+1}/\G\right)^* \cong \E_{h-1}^*$. 
\end{lemma}

\begin{proof}
  We note that \eqref{eq:estensione} yields
  \[ \mu(\F_{h+1})=\mu(\E'_1)=\mu(\E_h)=3(g-1),\]
  where $\mu$ as usual denotes the {\it slope} of a sheaf with respect to $\O_X(1)$. 
Assume that $\G$ is a destabilizing subsheaf of $\F_{h+1}$, that is  $0<\rk(\G) < \rk(\F_{h+1})=2(h+1)$ and $\mu(\G) \geq 3(g-1)$. Define
  \[ \Q:=\im\{\G \subset \F_{h+1} \to \E_h\} \; \; \mbox{and} \; \; \K:=\ker\{\G \to \Q\}.\]
  Then we may put \eqref{eq:estensione} into a commutative diagram with exact rows and columns:
  \begin{equation} \label{eq:estest}
    \xymatrix{ & 0 \ar[d] & 0 \ar[d] & 0 \ar[d] & \\
  0 \ar[r]   & \K \ar[d] \ar[r] & \G \ar[r] \ar[d] & \Q \ar[r] \ar[d] & 0 \\    
  0 \ar[r]   & \E'_1 \ar[r] \ar[d] & \F_{h+1} \ar[r] \ar[d] & \E_h \ar[d] \ar[r] & 0 \\
  0 \ar[r]   & \K' \ar[r] \ar[d] & \F_{h+1}/\G \ar[r] \ar[d] & \Q' \ar[r] \ar[d] & 0 \\
  & 0  & 0  & 0  &
}
\end{equation}
defining $\K'$ and $\Q'$. 

Assume that $\rk(\Q)=0$. Since $\E_h$ is torsion-free, we must have $\Q =0$, whence $\K \cong \G$. Since $\mu(\K) =\mu(\G) \geq 3(g-1)=\mu(\E'_1)$ and $\E'_1$ is slope--stable by assumption, we must have
$\rk(\K)=\rk(\E'_1)=2$. It follows that $\rk(\K')=0$.  Since
\[
  c_1(\K) =  c_1(\E'_1)-c_1(\K')=[\O_X(3)]-D',\]
where $D'$ is an effective divisor supported on the codimension--one locus of the support of $\K'$, we have 
\[ 3(g-1) \leq \mu(\K)=\frac{\left(\O_X(3)-D'\right) \cdot \O_X(1)^2}{2}= 3(g-1)-\frac{D' \cdot \O_X(1)^2}{2}.\]
  Hence $D'=0$, which means that $\K'$ is supported in codimension at least two.
  Thus, $\Shext^i(\K',\O_X)=0$ for $i \leq 1$, and it follows that
  $\G^* \cong \K^* \cong {\E_1'}^*$, as well as $\left(\F_{h+1}/\G\right)^* \cong {\Q'}^* \cong \E_h^*$,
as desired.

  Assume that $\rk(\K)=0$. Since $\E'_1$ is locally free, we must have $\K=0$, whence $\Q \cong \G$. Since $\mu(\Q) =\mu(\G) \geq 3(g-1)=\mu(\E_h)$ and $\E_h$ is slope--stable by assumption, we must have
$\rk(\Q)=\rk(\E_h)=2h$. It follows that $\rk(\Q')=0$. Since
\[
  c_1(\Q) =  c_1(\E_h)-c_1(\Q')=[\O_X(3h)]-D'',\]
where $D''$ is an effective divisor supported on the codimension--one locus of the support of $\Q'$, we have 
\[ 3(g-1) \leq \mu(\Q)=\frac{\left(\O_X(3h)-D''\right) \cdot \O_X(1)^2}{2h}= 3(g-1)-\frac{D'' \cdot \O_X(1)^2}{2}.\]
  Hence $D''=0$, which means that $\Q'$ is supported in codimension at least two.
  Thus, $\Shext^i(\Q',\O_X)=0$ for $i \leq 1$, and it follows that
  $\G^* \cong \Q^* \cong {\E_h}^*$. Hence, $\G^{**} \cong \E_h$, and \eqref{eq:estest} induces 
\[
\begin{tikzcd}
& & \G^{**} \arrow{d}  \isoarrow{rd} &  &  \\    
     0 \arrow{r}   & \E'_1 \arrow{r}  & \F_{h+1} \arrow{r}  & \E_h  \arrow{r} & 0,
\end{tikzcd}
\]
which gives a  splitting of \eqref{eq:estensione}, a contradiction.

We are therefore left with the case where $\rk(\Q)>0$ and $\rk(\K)>0$. 
Since $\E_1'$ and $\E_h$ are stable of slope $3(g-1)$ by hypothesis, we have
\begin{eqnarray}
  \label{eq:app9} \mu(\K) \leq 3(g-1), \; \; \mbox{with equality only if} \; \; \rk(\K)=\rk(\E'_1), \\
  \label{eq:app10} \mu(\Q) \leq 3(g-1), \; \; \mbox{with equality only if} \; \; \rk(\Q)=\rk(\E_h).
\end{eqnarray}
If equality holds in both \eqref{eq:app9}-\eqref{eq:app10}, we would have
\[ \rk (\G)=\rk(\K)+\rk(\Q)=\rk(\E'_1)+\rk(\E_h)=\rk(\F_{h+1}),\]
a contradiction to $\G \subset \F_{h+1}$ being a destabilizing sheaf. Hence, at least one of the inequalities in \eqref{eq:app9}-\eqref{eq:app10} must be strict, yielding
\begin{eqnarray*}
  3(g-1) &  \leq  & \mu(\G)=\frac{c_1(\G) \cdot \O_X(1)^2}{\rk(\G)} 
         =  \frac{\left(c_1(\K)+c_1(\Q)\right) \cdot \O_X(1)^2}{\rk(\K)+\rk(\Q)} \\
  & = & \frac{\mu(\K)\rk(\K)+\mu(\Q)\rk(\Q)}{\rk(\K)+\rk(\Q)}  <  \frac{3(g-1)\left(\rk(\K)+\rk(\Q)\right)}{\rk(\K)+\rk(\Q)}= 3(g-1),
\end{eqnarray*}
 a contradiction. 
\end{proof}

\begin{proposition} \label{prop:h+1-II}
  For all $h \in \ZZ^+$ the scheme $\mathfrak{\U}(h)$ is reduced of dimension
 $h^2(g+3)+1$ and its general member is slope--stable. 
\end{proposition}

\begin{proof}
 We prove this by induction on $h$, the case $h=1$ again being satisfied by the choice of $\UU_X(1)$. Assuming that we have proved the lemma for all integers $i\leq h$, we prove it for $h+1$.

The slope of the members of  $\UU_X(1)$ and $\UU_X(h)$ are
both equal to $3(g-1)$. It follows from a standard computation, see, e.g., \cite[Lemma 4.2]{ch}, that the general member $\UU_X(h+1)^{\rm ext}$ is simple. Hence also the general member $\E_{h+1}$ of $\UU_X(h+1)$ is simple, and it also satisfies $h^j(\E_{h+1} \* \E_{h+1}^*)=0$ for $j=2,3$ by Lemma \ref{lemma:indu}(iii).
 Therefore  
$\UU_X(h+1)$ is smooth at $[\E_{h+1}]$ (see, e.g., \cite[Prop. 2.10]{ch})
with
\begin{eqnarray} \label{eq:dimUh+1}
  \dim( \UU_X(h+1))&=&h^1(\E_{h+1} \* \E_{h+1}^*)=-\chi(\E_{h+1} \* \E_{h+1}^*)+h^0(\E_{h+1} \* \E_{h+1}^*)\\
\nonumber  &=&(h+1)^2(g+3)+1,
\end{eqnarray}
using the facts that $h^0(\E_{h+1} \* \E_{h+1}^*)=1$ as $\E_{h+1}$ is simple, and that
$\chi(\E_{h+1} \* \E_{h+1}^*)=-(h+1)^2(g+3)$ by Lemma \ref{lemma:indu}(iv). This proves that $\UU_X(h+1)$ is reduced of the claimed dimension. 

Finally, we prove that $\E_{h+1}$ is slope--stable.
Assume, to get a contradiction, that it is not.
 Then we may find a one-parameter family of bundles $\{\E^{(t)}\}$ over the disc $\DD$ such that $\E^{(t)}$ is a general member of $\UU_X(h+1)$ for $t \neq 0$ and $\E^{(0)}$ lies in $\UU_X(h+1)^{\rm ext}$, and such that we have a destabilizing sequence
\begin{equation} \label{eq:destat1} \xymatrix{
    0 \ar[r] & \G^{(t)} \ar[r] & \E^{(t)} \ar[r] & \F^{(t)} \ar[r] & 0}
  \end{equation}
  for $t \neq 0$, which we can take to be saturated, that is, such that $\F^{(t)}$ is torsion free, whence so that $\F^{(t)}$ and $\G^{(t)}$ are (Ulrich) vector bundles  (see \cite[Thm. 2.9]{ch} or \cite[(3.2)]{be2}).
  The limit of $\PP(\F^{(t)}) \subset \PP(\E^{(t)})$ defines a subvariety of $\PP(\E^{(0)})$ of the same dimension as $\PP(\F^{(t)})$, whence a coherent sheaf $\F^{(0)}$ of rank $\rk(\F^{(t)})$ with a surjection $\E^{(0)} \to \F^{(0)}$. Denoting by $\G^{(0)}$ its kernel, we have
$\rk(\G^{(0)})=\rk(\G^{(t)})$ and $c_1(\G^{(0)})=c_1(\G^{(t)})$. Hence, \eqref{eq:destat1} specializes to a destabilizing sequence for $t=0$. 
Lemma \ref{lemma:uniquedest1} yields  that ${\G^{(0)}}^*$ (resp.,
${\F^{(0)}}^*$) 
is the dual of a member of $\UU_X(1)$ (resp., of $\UU_X(h)$)
It follows that
${\G^{(t)}}^*$ (resp., ${\F^{(0)}}^*$)
is a deformation of the dual of a member of $\UU_X(1)$ (resp., $\UU_X(h)$), whence that $\G^{(t)}$ (resp., $\F^{(t)}$) is a deformation of a member of $\UU_X(1)$ (resp., $\UU_X(h)$), as both are locally free. 
It follows that 
$[\E_{h+1}] \in \UU_X(h+1)^{\rm ext}$. Thus,
\begin{equation} \label{eq:sonuguali}
  \UU_X(h+1)^{\rm ext}=\UU_X(h+1).
\end{equation}

On the other hand we have
\begin{equation} \label{eq:dimext2}
  \dim (\UU_X(h+1)^{\rm ext}) \leq \dim (\PP(\Ext^1(\E_{h},\E'_1))) +\dim (\UU_X(h)) +\dim (\UU_X(1)),
\end{equation}
for general 
  $[\E_{h}] \in \UU_X(h)$
and $[\E'_1]\in \UU_X(1)$. As $\E_{h}$ and  $\E'_1$ are slope--stable by induction, and of the same slope, we have $h^0(\E'_1 \* \E_h^*)=0$. Lemma \ref{lemma:indu}(i)--(ii) thus yields 
\[  h^1(\E'_1 \* \E_{h}^*)  =  -\chi(\E'_1 \* \E_{h}^*)+h^0(\E'_1 \* \E_{h}^*)+h^2(\E'_1 \* \E_{h}^*)-h^3(\E'_1 \* \E_{h}^*) 
=h(g+3).
\]
Hence, by \eqref{eq:dimext2} and \eqref{eq:dimUh+1} we have
\begin{eqnarray*}
  \dim (\UU_X(h+1)^{\rm ext}) & \leq & h(g+3) -1 +\left[h^2(g+3)+1\right]+\left[g+4\right] \\
  & = & (h^2+h+1)(g+3)+1  <  (h+1)^2(g+3)+1=\dim \UU_X(h+1),
\end{eqnarray*}
a contradiction to \eqref{eq:sonuguali}.
\end{proof}

This concludes the proof of Theorem \ref{thm:prFano}.

\section{ An application to curves with certain theta--characteristics}\label{sec:final}

In this section we make some remarks about rank 2 Ulrich bundles on prime Fano threefolds $X$ of genus $g$ and degree $2g-2$ in $\PP^{g+1}$. 
 By  \cite [Ex.  3.8]{CFK} any such $X$ carries a rank 2  Ulrich bundle $\E$ if and only if $\det(\E)=\O_X(3)$ and there is a smooth, irreducible, non--degenerate, linearly and quadratically normal curve $C$ in $X$ such that
$$
\deg(C)=5g-1, \quad g(C)=5g, \quad \omega_C=\O_C(2),
$$
whence $\O_C(1)$ is a {\em theta--characteristic} on $C$ with $h^0(\O_C(1))=g+2$. It is easy to check that such curves are automatically projectively Cohen--Macaulay.

More precisely, given a rank 2  Ulrich bundle $\E$ on $X$ and a general section $s\in H^0(\E)$, the curve $C$ is the zero locus of $s$, and conversely, if there is such a curve $C\subset X$, there is a rank 2  Ulrich bundle $\E$ on $X$ and a non--zero section $s\in H^0(\E)$, such that $C$ is the zero locus of $s$.

  Let $\mathfrak F_g$ be the Hilbert scheme  parametrizing (smooth) prime Fano threefolds $X$ of degree $2g-2$ in $\PP^{g+1}$, and $\mathfrak C_g$ be the  union of the components of the Hilbert scheme 
 containing points parametrizing curves $C$ in $\PP^{g+1}$ as above. 
Consider the incidence correspondence 
$$
I_g=\{(X,C)\in  \mathfrak F_g  \times \mathfrak C_g: C \subset X\}.
$$ 
Let 
$$
\phi_g: I_g\longrightarrow \mathfrak C_g
$$
be the projection to the second factor and consider the {\em moduli map}
$$
\mu_g: \im(\phi_g) \longrightarrow \mathcal M_{5g} 
$$where, as usual, $ \mathcal M_{p} $ denotes the moduli space of smooth 
curves of genus $p$. The image of $\mu_g$ is contained in the locus $ \mathcal M^{g+1}_{5g} $ of curves $C$ having a theta--characteristic 
$\theta$ with $\dim(|\theta|)=g+1$. 
According to a result by J. Harris (see \cite [Cor.\;1.11]{Harr}),  
any irreducible component of $ \mathcal M^{g+1}_{5g}$  in $ \mathcal M_{5g}$
has  \emph{expected codimension} 
\begin{equation}\label{eq:cod}
\frac {(g+2)(g+1)}2
\end{equation} 
and the actual codimension is at most this.

We want to discuss in this section the following question: is $ \mu_g $ dominant onto an irreducible component of $ \mathcal M^{g+1}_{5g} $? 

First we prove the following:

\begin{proposition}\label{prop:quest}  
For $g\geq 6$ the image of $ \mu_g $ has codimension larger than \eqref{eq:cod}  and therefore $ \mu_g $ is not dominant onto  
any  irreducible component of $ \mathcal M^{g+1}_{5g} $, or equivalently, the general curve in  any   component of 
$ \mathcal M^{g+1}_{5g}$ does not sit on a prime Fano threefold $X$ of  degree $2g-2$ in $\PP^{g+1}$ as a projectively Cohen--Macaulay semicanonical curve. 
\end{proposition}

\begin{proof} Let us compute the dimension of $I_g$.  We claim that the general  fiber of  the first projection $\psi_g: I_g\to  \mathfrak F_g$ has dimension $5g-1$. Indeed, given $X$ in $ \mathfrak F_g $, to give $C\subset X$ is equivalent to giving a bundle $\E$ in $\UU_X(1)$, as in \S\;\ref{sec:pf3}, and a non--zero section $s\in H^0(\E)$ up to a constant, i.e., an element in $\PP(H^0(\E))$. One has $\dim(\UU_X(1))=g+4$ by Theorem \ref {thm:prFano}, and $h^0(\E)=4(g-1)$ by \cite [(3.1)] {be2}. Hence the dimension of the fibers of $\psi_g$ is
$$
g+4+4(g-1)-1=5g-1
$$
as wanted. Hence 
\begin{equation}\label{eq:yto}
\dim(I_g)=\dim ( \mathfrak F_g )+5g-1 
\end{equation}
and 
$$
\dim(	\im(\phi_g))\leq \dim(I_g)=\dim ( \mathfrak F_g )+5g-1
$$
so that 
$$
\dim(\im(\mu_g))\leq\dim(	\im(\phi_g))-\dim({\rm PGL}(g+2,\mathbb C))\leq 
\dim ( \mathfrak F_g )+5g-1-\dim({\rm PGL}(g+2,\mathbb C)).
$$
Now $\dim ( \mathfrak F_g )-\dim({\rm PGL}(g+2,\mathbb C))$ is the number of moduli of the prime Fano threefolds of degree $2g-2$ in $\PP^{g+1}$, it is a well--known number (see \cite{FG}) and it is an easy computation to check that for $g\geq 6$ one has 
$$\dim ( \mathfrak F_g )+5g-1-\dim({\rm PGL}(g+2,\mathbb C))<\dim( \mathcal M_{5g} )-\frac {(g+2)(g+1)}2$$
hence 
$$\dim(\im(\mu_g))<\dim( \mathcal M_{5g} )-\frac {(g+2)(g+1)}2$$
which  proves the assertion by Harris' theorem mentioned above. 
\end{proof}

Next we complete the picture by proving that: 

\begin{proposition}\label{prop:quest1}  
For $3\leq g\leq 5$ the map $ \mu_g $ is dominant onto  a union of irreducible components  of 
$ \mathcal M^{g+1}_{5g} $,  each of which is uniruled and  of the expected codimension \eqref {eq:cod}.
\end{proposition}

\begin{proof} We prove the proposition only in case $g=5$. The cases $3\leq g\leq 4$, being similar, can be left to the reader.

One has $\dim ( \mathfrak F_5 )=75$ hence by \eqref {eq:yto} we have $\dim(I_5)=99$. Next we claim that:
\begin{itemize}
\item [(a)] $\phi_5$ is birational onto its image;
\item [(b)] any component of  
the image of $\phi_5$ is a component of the sublocus of $\mathfrak C_5$ parametrizing curves whose hyperplane section is a theta--characteristic;
\item [(c)] the  fibers  of the map $\mu_5$ are finite, modulo the action of ${\rm PGL}(7,\mathbb C)$.
\end{itemize}

To prove (a), note that, given $C$, the quadrics of $\PP^6$ cut out on it the complete canonical series, which has dimension $g(C)-1=24$. Since $\dim(|\O_{\PP^6}(2)|)=27$, one has $\dim (|\mathcal I_{C/\PP^6}(2)|)=2$, hence $C$ sits on a unique prime Fano threefold $X$ of degree $8$ in $\PP^6$. 

To prove (b), denote by $\mathfrak K$  any component  of the sublocus of $\mathfrak H_5$ parametrizing curves whose hyperplane section is a theta--characteristic 
 and which contains a component of $\im(\phi_5)$. If $C$ is a general element of $\mathfrak K$,  by semi--continuity  
$C$ is still projectively Cohen--Macaulay, hence one sees, as above, that $\dim (|\mathcal I_{C/\PP^6}(2)|)=2$, so $C$ sits on a prime Fano threefold 
$X$ of degree $8$ in $\PP^6$, thus $C$ in $ \im(\phi_5)$ as wanted. 

Finally, to prove (c), suppose that a  fiber   of $\mu_5$ has a positive dimensional component $\mathcal C$. Every element $C$ in $\mathcal C$ is endowed with its hyperplane section bundle $\O_C(1)$, which is a theta--characteristic  $\theta$ with $h^0 (\theta) = g+2$. Fix $C$  and take a general element $C'$ in $ \mathcal C$. Since $C$ and $C'$ are in the same fiber of $\mu_5$, they are isomorphic. The isomorphism that maps $C'$ to $C$ has to map the theta--characteristic $\O_{C'}(1)$ to a very ample theta--characteristic $\theta$ of $C$ with $h^0(\theta)=g+2$. Since $C$ does not have a continuous family of theta--characteristics, the theta--characteristic $\theta$ has to be independent of $C'$  in the fiber 
$\mathcal C$  which shows that all curves in $\mathcal C$ are projectively equivalent.

By (a) above  we have  $\dim(\im(\phi_5))=99$. By (c) one has 
$$\dim(\im(\mu_5))=99-\dim({\rm PGL}(7,\mathbb C))=51.$$
By (b),  any component of  $\im(\mu_5)$ is an irreducible component 
of $ \mathcal M_{25}^{6} $ of codimension
$$
\dim(\mathcal M_{25})-\dim (\im(\mu_5))=72-51=21,
$$
which is the expected codimension \eqref{eq:cod} for $g=5$, as wanted.

Finally,  each component of  $\im(\mu_5)$ is uniruled because so is $I_5$. 
\end{proof}

\end{document}